\documentclass[11pt,a4paper]{amsart}
\usepackage{amsfonts}
\usepackage{amsthm}
\usepackage{amsmath}
\usepackage{amscd}
\usepackage[latin2]{inputenc}
\usepackage{t1enc}
\usepackage[mathscr]{eucal}
\usepackage{indentfirst}
\usepackage{graphicx}
\usepackage{graphics}
\usepackage{pict2e}
\usepackage{epic}
\usepackage{amssymb}
\usepackage[colorlinks,citecolor=red,urlcolor=blue,bookmarks=false,hypertexnames=true]{hyperref} 
\numberwithin{equation}{section}
\usepackage[margin=3cm]{geometry}

\theoremstyle{plain}
\newtheorem{Th}{Theorem}[section]
\newtheorem{Lemma}[Th]{Lemma}
\newtheorem{Cor}[Th]{Corollary}
\newtheorem{Prop}[Th]{Proposition}
\newtheorem*{Theorem-non}{Theorem}
\newtheorem*{Theorem-non2}{Theorem}

 \theoremstyle{definition}
 \newtheorem*{Proof-non}{Proof of Theorem \ref{Theorem 1.1} assuming Proposition \ref{Prop1} and Proposition \ref{Propm}}
\newtheorem*{Proof-non2}{Proof of (1)  ($\bf{m_{1}}$-estimate) in Proposition \ref{Propm} assuming Proposition \ref{Proposition 5.1}}
\newtheorem*{Proof-non3}{Proof of (2)  ($\bf{m_{2}}$-estimate) in Proposition \ref{Propm}}
\newtheorem*{Proof-non4}{Proof of Proposition \ref{Prop1}}
\newtheorem{Def}[Th]{Definition}

\newtheorem{Rem}[Th]{Remark}
\newtheorem{?}[Th]{Problem}

\newcommand{\RomanNumeralCaps}[1]
    {\MakeUppercase{\romannumeral #1}}
\begin{document}

\email{Jiseongk@buffalo.edu}
\address{University at Buffalo, Department of Mathematics
244 Mathematics Building
Buffalo, NY 14260-2900}
\title{On the asymptotics of the shifted sums of hecke eigenvalue squares}
\author{jiseong kim}
\begin{abstract} 
The purpose of this paper is to obtain asymptotics of shifted sums of Hecke eigenvalue squares on average. We show that for $X^{\frac{2}{3}+\epsilon} < H <X^{1-\epsilon},$ there are constants $B_{h}$ such that 
$$ \sum_{X\leq n \leq 2X} \lambda_{f}(n)^{2}\lambda_{f}(n+h)^{2}-B_{h}X=O_{f,A,\epsilon}\big(X (\log X)^{-A}\big)$$ for all but $O_{f,A,\epsilon}\big(H(\log X)^{-3A}\big)$ integers $h \in [1,H]$  where $\{\lambda_{f}(n)\}_{n\geq1}$ are normalized Hecke eigenvalues of a fixed holomorphic cusp form $f.$ 
%With the generalized Lindel\"of hypothesis, we also proved that $G_{h}(X) =o(X)$ for almost all $h \in [1,H],$
%when $H$ is sufficiently large(near to $X.$)
Our method is based on the Hardy-Littlewood circle method. We divide the minor arcs into two parts $m_{1}$ and $m_{2}.$ In order to treat $m_{2},$ we use the Hecke relations, a bound of Miller to apply some arguments from a paper of Matom\"a{}ki, Radziwi\l{}\l{} and Tao. We apply Parseval's identity and Gallagher's lemma so as to treat $m_{1}.$

\end{abstract}

\maketitle
\tableofcontents 
\section{\rm{Introduction}}
\noindent Let $f(z)$ be a holomorphic Hecke cusp form of an even integral weight $\kappa$
 for the full modular group $SL(2,\mathbb{Z}).$ It is known that every holomorphic Hecke cusp form has a Fourier expansion at the cusp $\infty.$ Therefore, $f(z)$ has a Fourier expansion
\begin{equation}
    f(z)=\sum_{n=1}^{\infty} c_{n}n^{\frac{\kappa-1}{2}}e(nz)
\end{equation}
for some $\{c_{n}\}_{n \geq 1} \subset \mathbb{R}$ in which $e(z)=e^{2\pi iz}$. For each $n \geq 1,$
$$T_{n}f(z)=\frac{1}{n} \sum_{ad=n} a^{\kappa} \sum_{0\leq b <d} f(\frac{az+b}{d})=\lambda_{f}(n) f(z)$$ where $\{T_{n}\}_{n\geq 1}$ are the $n$-th Hecke operators, $\{\lambda_{f}(n)\}_{n\geq 1} \subset \mathbb{R}$ are the normalized Hecke eigenvalues of $f.$ The Hecke eigenvalues $\{\lambda_{f}(n)\}_{n \geq 1}$ satisfy the following properties:

\begin{equation}
c_{1}\lambda_{f}(n)=c_{n},
\end{equation}
\begin{equation}
    \lambda_{f}(m) \lambda_{f}(n)= \sum_{d|(n,m)} \lambda_{f}(\frac{nm}{d^{2}}),
\end{equation}

\begin{equation} \label{divisor bound}
     |\lambda_{f}(n)| \leq d_{2}(n),
\end{equation}
  in which $d_{2}(n):= \sum_{m|n} 1,$ and the inequality \eqref{divisor bound} is called the Deligne bound.
We refer the reader to \cite[Chapter 14]{IK1} for more details.
There are some results about asymptotic evaluations of sums related to Hecke eigenvalues. Rankin \cite{RA1} and Selberg \cite{S} showed that 
\begin{equation}\label{Rankin-Selberg}
    \sum_{1 \leq n \leq X}\lambda_{f}(n)^{2}=c_{1,f}X+O_{f}(X^{\frac{3}{5}})
\end{equation} for some constant $c_{1,f}.$  Recently, Huang \cite{huang2021rankinselberg} improved the error term $O_{f}(X^{\frac{3}{5}})$ to $O_{f}(X^{\frac{3}{5}-\frac{1}{560}+o(1)}).$
 Fomenko \cite{FM} showed that
\begin{equation}\label{Fomenko}
\sum_{1\leq n\leq X} \lambda_{f}(n)^{4} =c_{2,f}X\log X + d_{f}X+O_{f,\epsilon}(X^{\frac{9}{10}+\epsilon})
\end{equation} for some constants $c_{2,f}, d_{f}.$
For results of higher powers, we refer the reader to \cite{L1}. Since the divisor function appears in a few identities, such as relations between the Riemann zeta function and the Eisenstein series, the additive divisor problems can be considered as prototypes of the shifted convolution problems for coefficients of automorphic forms. 
\noindent Let $d_{k}(n):= \sum_{n_{1},n_{2},...n_{k} \in \mathbb{N} \atop n_{1}n_{2}...n_{k}=n} 1.$
Deshouillers, Iwaniec \cite{DeshIwan} proved that 
$$\sum_{X \leq n \leq 2X} d_{2}(n)d_{2}(n+h)=P_{2}(\log X)X +O(X^{\frac{2}{3}+o(1)})$$
as $X \rightarrow \infty$ where $P_{2}(x)$ is a quadratic polynomial with coefficients depending on $h.$  Baier, Browning, Marasingha and Zhao \cite{BBMZ} proved the following theorem.
\begin{Theorem-non2} Let $\epsilon$ be fixed small positive constant. Assume that $X^{\frac{1}{3}+\epsilon} \leq H \leq X^{1-\epsilon}.$ Then there exists $\delta>0$ such that for all but $O_{\epsilon}(HX^{-\delta})$ integers $|h|<H,$ $$\sum_{X \leq n \leq 2X} d_{3}(n)d_{3}(n+h)=P_{3}(\log X)X + O(X^{1-\delta})$$
where $P_{3}(x)$ is a quartic polynomial with coefficients depending on $h.$ 
\end{Theorem-non2} \noindent For the higher order divisor functions, Matom\"a{}ki, Radziwi\l{}\l{} and Tao proved the following averaged divisor correlation conjecture \cite{MRT1}.
\begin{Theorem-non}
\cite[Theorem 1.3, (ii)]{MRT1}
Let $A>0,$ and let $0<\epsilon< \frac{1}{2}.$ Let $k,l \geq 2$ be fixed. Suppose that $X^{\frac{8}{33}+\epsilon} \leq H \leq X^{1-\epsilon}$ for some $X\geq 2.$
Let $0\leq h_{0} \leq X ^{1-\epsilon}$.
Then for each $h,$ there exists a polynomial $P_{k,l,h}$ of degree $k+l-2$ such that
$$\sum_{X \leq n \leq 2X} d_{k}(n)d_{l}(n+h) = P_{k,l,h}(\log X)X + O_{A,\epsilon, k,l}\big(X(\log X)^{-A}\big) \; \textrm{as} \; X \rightarrow \infty$$
for all but $O_{A,\epsilon,k,l}\big(H (\log X)^{-A}\big)$ values of $h$ with $|h-h_{0}| \leq H.$ 
\end{Theorem-non} \noindent In \cite{MRT2}, it was shown that for $k\geq l,$ $(\log X)^{10000k\log k} \leq H \leq X^{1-\epsilon},$ there are constants $C_{k,l,h}>0$ such that  
\begin{equation}
\sum_{1\leq |h| \leq H} \Big| \sum_{X \leq n \leq 2X} d_{k}(n)d_{l}(n+h)-C_{k,l,h}X (\log X)^{k+l-2} \Big| = o\big(HX (\log X)^{k+l-2}\big).
\end{equation}
\noindent In this paper, as an analog of the divisor correlation problems, we study asymptotic estimations of the shifted convolution sums of the Hecke eigenvalue squares.

\subsection{Main results}

\begin{Th}\label{Theorem 1.1}Let $\epsilon$ be a fixed small positive constant. Let $A \in \mathbb{N}.$ Suppose $X^{\frac{2}{3}+\epsilon} \leq H\leq X^{1-\epsilon}.$ There are constants $B_{h}$ such that as $X \rightarrow \infty,$
\begin{equation}\label{1theorem1} \sum_{X \leq n\leq 2X} \lambda_{f}(n)^{2}\lambda_{f}(n+h)^{2}-B_{h}X=O_{f,A,\epsilon}\big(X (\log X)^{-A}\big) \end{equation}
for all but $O_{f,A,\epsilon}\big(H(\log X)^{-3A}\big)$ integers $h \in [1,H].$ 

\end{Th}
\noindent Precisely, we get 
\begin{equation}
B_{h} = \sum_{q=1}^{\infty} \thinspace  \sum_{1\leq a \leq q \atop(a,q)=1} \big(\sum_{q=q_{0}q_{1}}\frac{\mu(q_{1})}{\phi(q_{1})q_{0}}w_{f,q_{0},q_{1}}\big)^{2}e(h\frac{a}{q}), 
\end{equation}
where $$w_{f,q_{0},q_{1}} = c_{1,f}\prod_{p|q} (\frac{p-1}{p+1})\big(1-\frac{\lambda_{f}(p)^{2}-2}{p}+\frac{1}{p^{2}}\big) \prod_{p^{l} \parallel q_{0}}\big(\lambda_{f}(p^{l})^{2}+\lambda_{f}(p^{l+1})^{2}p^{-1}+...\big)$$  here $\phi$ is the Euler totient function, $\mu$ is the M\"{o}bius function, and $p^{l} \parallel q_{0}$ means $p^{l} | q_{0},$ but $p^{l+1} \nmid q_{0}$ (see Remark \ref{Remark 3.3}, Remark \ref{Remark 3.5}). By \eqref{3dq}, one can show that $$B_{h}\ll_{f} 1+ \sum_{a|h \atop a \neq 1} \frac{d_{2}(a)^{2}(\log a)^{14}}{a}\ll_{f}d_{2}(h).$$

\begin{Rem}\label{Remark 1.2}
It is known that for any $1<X<Y,$ summing $\frac{\lambda_{f}(p)^{2}}{p}$ over primes
\begin{equation}\label{squareprime} \sum_{ X \leq p \leq Y} \frac{\lambda_{f}(p)^{2}}{p}= \sum_{X \leq p \leq Y} \frac{1}{p} +O_{f}(1)\end{equation}(see \cite{MR} Lemma 3.1,(iii)).
Since the improved error term for \eqref{Rankin-Selberg} is $O(X^{\frac{3}{5}-\frac{1}{560}+o(1)}),$ for any $y \in \mathbb{N}$ and all $x \in [X,2X],$ $$\sum_{x \leq n \leq x+y} \lambda_{f}(n)^{2}=c_{1,f}y+O_{f}(X^{\frac{3}{5}-\frac{1}{560}+o(1)}).$$
Interchanging the order of summations, we see that for $X^{\frac{3}{5}-\frac{1}{560}+o(1)} \leq H \leq X^{1-\epsilon},$

\begin{equation}\begin{split}\sum_{1 \leq h \leq H} \sum_{X \leq n\leq 2X} \lambda_{f}(n)^{2}\lambda_{f}(n+h)^{2} &= \sum_{X\leq n \leq 2X} \lambda_{f}(n)^{2} \sum_{1\leq h \leq H} \lambda_{f}(n+h)^{2} \\&=  c_{1,f}^{2}HX+O_{f}(X^{\frac{8}{5}-\frac{1}{560}+o(1)})\end{split} \nonumber \end{equation} where $c_{1,f}$ is the coefficient in \eqref{Rankin-Selberg}.

\end{Rem}

\subsection{Sketch of proof}
By the Hardy-Littlewood circle method, each shifted sum can be expressed as an integral over the unit interval. Let $Q=(\log X)^{B}$ for some $B>0.$ By the Dirichlet approximation theorem, it is known that for any $\alpha \in [0,1],$ there is an interval  $\big(\frac{a}{q}-\frac{1}{M_{1}(q,Q)},\frac{a}{q}+\frac{1}{M_{2}(q,Q)}\big)$ containing $\alpha$ such that $1 \leq a \leq q\leq Q , \;(a,q)=1,$ and $qQ \leq \min \big(M_{1}(q,Q),M_{2}(q,Q)\big).$ In Section 2 we divide the interval [0,1] into major arcs and minor arcs. Each major arc is the set of $\alpha$ such that  $$|\alpha- \frac{a}{q}| \leq X^{-\frac{5}{6}-2\epsilon}$$
for some $q\leq (\log X)^{B}, \;(a,q)=1.$       
In Section 3 we treat the major arcs. From the major arcs, we obtain the main term $B_{h}X$ in our asymptotic formula.  
In Section 4 and Section 5 we treat the minor arcs. %For the minor arc estimates, we decompose $\lambda_{f}(n)^{2}$ by the Hecke relations. By applying some upper bounds of summations of $\lambda_{f}(n^{2})$(a bound of Miller), we apply some parts of  methods in \cite{MRT1}. 
In \cite{MRT1} authors separated the minor arcs into two parts
 \begin{equation} \begin{split}
&m_{1}:= \bigcup_{1 \leq q \leq Q} \bigcup_{1 \leq a \leq q \atop (a,q)=1} [\frac{a}{q}-X^{-\frac{1}{6}-\epsilon}, \frac{a}{q}-\frac{(\log X)^{B'}}{X}] \cup [\frac{a}{q}+\frac{(\log X)^{B'}}{X}, \frac{a}{q}+X^{-\frac{1}{6}-\epsilon}],\\&
m_{2}:= \bigcup_{1 \leq q \leq Q} \bigcup_{1\leq a \leq q \atop (a,q)=1} [\frac{a}{q}-\frac{1}{M_{1}(q,Q)}, \frac{a}{q}-X^{-\frac{1}{6}-\epsilon}] \cup [\frac{a}{q}+X^{-\frac{1}{6}-\epsilon}, \frac{a}{q}+\frac{1}{M_{2}(q,Q)}]
\end{split} \nonumber
\end{equation}
for some $B'$ depending on $B.$
The authors treated $m_{1}$ using the Huxley large value estimates and the Dirichlet mean value theorems. For $m_{2},$ the authors applied some techniques of decomposing the divisor functions as Dirichlet convolutions and further dyadic decomposition. 
After decomposing the divisor functions to some types of sums,  the authors applied the following theorem.
\begin{Theorem-non}\cite[Proposition 6.1]{MRT1}
Let  $\epsilon>0$  be sufficiently small. 
Let  $A>0$  be fixed, and let  $B>0$  be sufficiently large depending on  $A$. 
Let  $X\geq 2$ , and set  $H:= X^{\frac{8}{33} +\epsilon}$. 
Set  $Q=(\log X)^{B}$,  $\rho=Q^{-\frac{1}{2}}$, and let  $1\leq q_{1}\leq Q.$ 
Let  $\gamma$  be a quantity such that  $X^{-\frac{1}{6}-2\epsilon}\leq \gamma \leq \frac{1}{q_{1}Q}$.
Let  $f:\mathbb{N}\rightarrow \mathbb{C}$  be a function of Type $\rm{d_{1}}$ or Type $\rm{\RomanNumeralCaps{2}}$ sum.
Then 
\begin{equation} \int_{\rho \gamma X \ll t \ll \rho^{-1} \gamma X} \Big(\sum_{\chi (\rm{mod} \thinspace q_{1})} \int_{t-\gamma H}^{t+\gamma H} \big| D[f](\frac{1}{2}+it',\chi)\big|dt'\Big)^{2}dt \end{equation} $$\ll_{\epsilon,A,B}q_{1}|\gamma|^{2}H^{2}X(\log X)^{-A},$$
\end{Theorem-non}where $$D[f](s,\chi):= \sum_{n=1}^{\infty} \frac{f(n)\chi(n)}{n^{s}} .$$
In this paper, we also separate the minor arcs into two parts, but $m_{1},m_{2}$ are quite different compared to the above separation. In Section 4 we basically follow the methods in \cite{MRT1} for  $m_{2},$ but we use the Hecke relations to decompose $\lambda_{f}(n)^{2}$ as $\lambda_{f}(m^{2}) \ast 1.$ For the divisor correlation problems, $d_{k}(n)=1 \ast 1 \ast ... \ast 1.$ It is easy to check that $$\sum_{n\leq x \atop n \equiv a \: (\mathrm{mod} \:\it{q})} \frac{1}{n^{\frac{1}{2}+it}} \ll_{\mathcal{A},\mathcal{B},\mathcal{B}'} x^{\frac{1}{2}}(\log X)^{-\mathcal{A}}$$
for any $\mathcal{A},\mathcal{B},\mathcal{B}'>0,\: 1\leq a \leq  q \leq (\log X)^{\mathcal{B}} \; \textrm{and}\; t \in [(\log x)^{\mathcal{B}'}, x^{\mathcal{B}'}]$ where $\mathcal{B}'$ is sufficiently large depending on $\mathcal{A},\mathcal{B}$ (we say that $1$ has good cancellation. We refer the reader to \cite[Subsection 2.4]{MRT1}). The above inequality contributes to the $(\log X)^{-A}$ cancellation in (1.11). Also, the supports of the characteristic functions 1 in the Dirichlet convolution can be interchanged, and it is beneficial for their $m_{2}$ estimates (there are some conditions on the support of each factor. For details, see \cite[Lemma A.1]{MRT1}).  In our case we have $\lambda_{f}(m^{2})$ and $1$ in the Dirichlet convolution, so we can not interchange their supports. Instead, we treat a perturbation $b(m)$ of $\lambda_{f}(m)^{2}$ such that 
$$b(m):= \sum_{d|m \atop 1\leq d < HX^{-\epsilon^{2}} }\lambda_{f}(d^{2}).$$
To bound a difference between $b_{m}$ and $\lambda_{f}(m)^{2},$ we apply the following Lemma.  

\begin{Lemma}\label{Lemma 1.3}$($\cite[Theorem 1.1]{Miller1}$,\rm{\;A\; bound\; of\; Miller})$
For any $\epsilon>0$ and any $\alpha \in \mathbb{R},$
\begin{equation}
    \sum_{1\leq n\leq X} \lambda_{f}(n^{2})e(\alpha n) \ll_{f,\epsilon} X^{\frac{3}{4}+\epsilon}.
\end{equation}

\end{Lemma}
\begin{Rem}\label{Remark 1.4} %Due to the application on Lemma 1.3 on \eqref{4pointwise}, we need  $X^{\frac{2}{3}+\epsilon} \leq H.$
The condition $X^{\frac{2}{3}+\epsilon} \leq H$ in Theorem \ref{Theorem 1.1} arises from an application of Lemma \ref{Lemma 1.3} on \eqref{4pointwise}. Precisely, by taking $Y=HX^{-\epsilon^{2}},$ $H$ should satisfy the following inequality
\begin{equation}\label{1errorterm1}X^{2}(HX^{-\epsilon^{2}})^{-\frac{1}{2}+2\epsilon}H^{-1} =X^{2+\frac{\epsilon^{2}}{2}-2\epsilon^{3}}H^{-\frac{3}{2}+2\epsilon}=O_{f,\epsilon}(X^{1-\epsilon}).\end{equation} If one assumes a square root cancellation such that 
$$\sum_{1\leq n\leq X} \lambda_{f}(n^{2})e(\alpha n) \ll_{f,\epsilon} X^{\frac{1}{2}+\epsilon},$$ then by the similar methods of Lemma \ref{Lemma 4.1}, one can get 
$$X^{2}(HX^{-\epsilon^{2}})^{-1+2\epsilon}H^{-1} = O_{f,\epsilon}(X^{1-\epsilon})$$ instead of \eqref{1errorterm1}, and the condition on $H$ can be improved to $X^{\frac{1}{2}+\epsilon} \leq H.$
\end{Rem}
\noindent 
For the range $X^{\frac{5}{6}+\epsilon} \leq H \leq X^{1-\epsilon},$ we only need the methods used to estimate the contribution of $m_{2}.$ In Section 5 we develop a new method which originates from Gallagher's Lemma \eqref{Gallagher}, and we apply it to treat $m_{1}$. For applications, this method requires some mean square upper bounds of corresponding Dirichlet series on the line $\Re(s)=\frac{1}{2}.$ Therefore, we expect that this method can be applied to more general multiplicative functions.
\begin{Rem}\label{Remark 1.5} The length of the major arc intervals arises from the application of Lemma \ref{Lemma 3.4}, Lemma \ref{Lemma 3.7} on Proposition \ref{Prop1}. Roughly speaking, error terms of the sums of the Hecke eigenvalue squares twisted by various non-principal characters modulo $q$ are bounded by $O_{f,q,\epsilon}(X^{\frac{3}{4}+\epsilon}),$ and these bounds are related to error terms of the shifted sums. To make these error terms negligible, the length of each interval should be bounded by $X^{-\frac{5}{6}-2\epsilon}.$ Since we have better upper bounds for Lemma \ref{Lemma 3.4} and Lemma \ref{Lemma 3.7} (see Remark \ref{Remark 3.6}), one may treat wider ranges of intervals for the major arcs. Since the length $X^{-\frac{5}{6}-2\epsilon}$ is compatible with methods in Section 5, we apply the upper bounds in Lemma \ref{Lemma 3.4}, Lemma \ref{Lemma 3.7}.
\end{Rem}
\begin{Rem}\label{Remark 1.6}
If we assume the Ramanujan conjecture on Maass cusp forms over $SL(2,\mathbb{Z}),$ then our results can be generalized to Maass cusp form over $SL(2,\mathbb{Z}).$
To generalize our results to Maass cusp forms over $SL(m, \mathbb{Z})$ for $m \geq 3,$ with the generalized Ramanujan conjecture, we need some theorems analogous to Lemma \ref{Lemma 1.3}. For any fixed Hecke Maass cusp form $G$ over $SL(m, \mathbb{Z}),$ it is known that 
$$|\Lambda_{G}(n,1,...,1)|^{2}= \sum_{d|n} \Lambda_{G}(d,1,...,1.d)$$ where $\{\Lambda_{G}(n,1,...,1)\}_{n\geq 1}$ are the normalized Hecke eigenvalues of $G$. Therefore, we need an analog of Lemma \ref{Lemma 1.3} such that 
\begin{equation}\sum_{1 \leq d \leq X} \Lambda_{G}(d,1,...,1,d)e(d\alpha)=O_{\epsilon}(X^{\delta+\epsilon})\end{equation} for some $\delta<1.$ But as far as the author is aware, it does not yet exist in the literature for $m \geq 3.$ If one has $\delta \leq \frac{3}{4},$  then one can apply our methods for $m_{2}$ estimates (so the lower bound of $H$ will be $X^{\frac{5}{6}+\epsilon}$).

\end{Rem}
\subsection{Corollary}
This subsection is devoted to study an average version of Theorem \ref{Theorem 1.1}. Let us start to give an explicit upper bound on \eqref{1theorem1}.
\begin{Lemma}\label{Lemma 1.7}Let \begin{equation}S_{f}(\alpha):=\sum_{X \leq n \leq 2X} \lambda_{f}(n)^{2}e(n\alpha)\end{equation} for $\alpha \in \mathbb{R}.$
Then for any $h \neq 0,$
\begin{equation} \int_{0}^{1} |S_{f}(\alpha)|^{2}e(h\alpha)d\alpha \ll_{f} X (\log \log h)^{16}.
\end{equation}
\end{Lemma}
\begin{proof}
By the Deligne bound, summing over prime 
$$\sum_{1<p<X}\sum_{k=2}^{\infty} \frac{\lambda_{f}(p^{k})^{2}}{p^{k}} = O(1).$$ 
Applying Shiu's Theorem (see \cite{MRT2}, Lemma 2.3,(ii)) and Remark \ref{Remark 1.2}, we see that 
\begin{equation}\label{integralsquare}
\begin{split}
 \int_{0}^{1} &|S_{f}(\alpha)|^{2}e(h\alpha)d\alpha \\&\ll  \sum_{X \leq n,n+h \leq 2X} \lambda_{f}(n)^{2}\lambda_{f}(n+h)^{2}
\\&\ll \prod_{p | h} (1+\frac{2\lambda_{f}(p)^{2}}{p}+\frac{\lambda_{f}(p)^{4}}{p}) X\prod_{1<p\leq X}(1-\frac{2}{p})\sum_{n_{1}n_{2} \leq X \atop (n_{1}n_{2},h)=1} \frac{\lambda_{f}(n_{1})^{2} \lambda_{f}(n_{2})^{2}}{n_{1}n_{2}} \\&  \ll \prod_{p | h} (1+\frac{2\lambda_{f}(p)^{2}}{p}+\frac{\lambda_{f}(p)^{4}}{p})(1+\frac{\lambda_{f}(p)^{2}}{p})^{-2} X\prod_{1<p \leq X}(1-\frac{2}{p})  (1+\frac{\lambda_{f}(p)^{2}}{p})^{2} \\   &\ll  X\prod_{p|h}(1+\frac{\lambda_{f}(p)^{4}}{p})\prod_{1<p \leq X}(1-\frac{2}{p})  (1+\frac{\lambda_{f}(p)^{2}}{p})^{2}.
\end{split}
\end{equation}
By \eqref{squareprime}, we have $$\log \Big(\prod_{1<p\leq X}(1+\frac{2\lambda_{f}(p)^{2}-2}{p})\Big) = \sum_{1<p \leq X} \log (1+\frac{2\lambda_{f}(p)^{2}-2}{p}) =O_{f}(1).$$
Therefore, the last term in \eqref{integralsquare} is bounded by $$O_{f}\big(X\prod_{p|h}(1+\frac{\lambda_{f}(p)^{4}}{p})).$$
Using the Deligne bound and the prime number theorem, we see that 
\begin{equation}\begin{split}\prod_{p|h}(1+\frac{\lambda_{f}(p)^{4}}{p}) &\ll \prod_{p \leq \log h}(1+\frac{2^{4}}{p}) \\&\ll (\log \log h)^{16}. \end{split} \nonumber \end{equation}
\end{proof}
\noindent Applying Theorem \ref{Theorem 1.1} and Lemma \ref{Lemma 1.7}, we prove the following corollary.
\begin{Cor}\label{Corollary 1.8}Let $\epsilon$ be a fixed small positive constant and let  $A \in \mathbb{N}.$ Suppose $X^{\frac{2}{3}+\epsilon} \leq H\leq X^{1-\epsilon}.$ There are constants $B_{h}$ such that as $X \rightarrow \infty,$
\begin{equation} \sum_{1\leq h \leq H}  |\sum_{X \leq n\leq 2X} \lambda_{f}(n)^{2}\lambda_{f}(n+h)^{2}-B_{h}X|=O_{f,A.\epsilon}(HX (\log X)^{-A}). \end{equation}

\end{Cor}
\begin{proof}
By Lemma \ref{Lemma 1.7}, 
\begin{equation}\label{1shifted}
\begin{split}
 \sum_{X \leq n \leq 2X} \lambda_{f}(n)^{2}\lambda_{f}(n+h)^{2}-B_{h}X &\ll_{f} X \max(|B_{h}|, (\log \log h)^{16})\\&\ll_{f} X\max(d_{2}(h),(\log\log h)^{16})
\end{split}
\end{equation}
 for all $h \in [1,H].$ By Theorem \ref{Theorem 1.1}, the left-hand side of \eqref{1shifted} is bounded by
 $O_{f,A,\epsilon}(X (\log X)^{-A})$ except for $O_{f,A,\epsilon}(H (\log X)^{-3A})$ many values of $h \in [1,H].$
Since \begin{equation}\begin{split}\sum_{1 \leq h \leq H} d_{2}(h) &\ll H\log H \\&\ll H\log X,\end{split}\nonumber\end{equation} 
\begin{equation}\sum_{1\leq h \leq H}  |\sum_{X \leq n\leq 2X} \lambda_{f}(n)^{2}\lambda_{f}(n+h)^{2}-B_{h}X|\ll_{f,A,\epsilon} HX(\log X)^{-3A+1}+ HX (\log X)^{-A}.\nonumber\end{equation}
\end{proof}
\section{\rm{Notation and the circle method}}
\subsection{Notation}
Throughout this paper, we assume that $X$ is sufficiently large. 
For any two functions $k(x):\mathbb{R} \rightarrow \mathbb{R}$ and $l(x):\mathbb{R} \rightarrow \mathbb{R}^{+}$, we use $k(x)\ll l(x)$ or  $k(x)=O(l(x))$ to denote that there exists a positive constant $C$ such that  $|k(x)| \leq C l(x)$ for all $x.$ For any set $A,$ we use $1_{A}$ to denote the characteristic function on $A.$ Summing over the index $p$ denotes summing over primes. $\log_{2}(x)$ denotes the binary logarithm.
\noindent Let
\begin{align*}
&Q=(\log X)^{B}  \rm{ \enspace for \enspace some \enspace} \it{B}>\rm{10},\\
&I  := [0,1],\\
& \mathcal{M} :=\bigcup_{1 \leq q \leq Q}  \bigcup_{1 \leq a \leq q\atop (a,q)=1} (\frac{a}{q}-X^{-\frac{5}{6}-2\epsilon} , \frac{a}{q}+X^{-\frac{5}{6}-2\epsilon}),\\
& m := \mathcal{M} \setminus I := I \cap \mathcal{M}^{c}, \\
& m_{1}:= \bigcup_{1 \leq q \leq Q} \bigcup_{1 \leq a \leq q \atop (a,q)=1} [\frac{a}{q}-\frac{1}{2}X^{-\frac{1}{2}-\epsilon}, \frac{a}{q}-X^{-\frac{5}{6}-2\epsilon}] \cup [\frac{a}{q}+X^{-\frac{5}{6}-2\epsilon}, \frac{a}{q}+\frac{1}{2}X^{-\frac{1}{2}-\epsilon}],\\
& m_{2}:= \bigcup_{1 \leq q \leq Q} \bigcup_{1 \leq a \leq q \atop (a,q)=1} [\frac{a}{q}-\frac{1}{M_{1}(q,Q)}, \frac{a}{q}-\frac{1}{2}X^{-\frac{1}{2}-\epsilon}] \cup [\frac{a}{q}+\frac{1}{2}X^{-\frac{1}{2}-\epsilon}, \frac{a}{q}+\frac{1}{M_{2}(q,Q)}],\\
&\textrm{Res}_{s=a}f(s):= \textrm{The \thinspace residue \thinspace of\thinspace}f \textrm{\thinspace at \thinspace} s=a,\\ 
&B_{h}  :=\sum_{q=1}^{\infty} \thinspace  \sum_{1\leq a\leq q \atop (a,q)=1} (\sum_{q=q_{0}q_{1}}\frac{\mu(q_{1})}{\phi(q_{1})q_{0}}w_{f,q_{0},q_{1}})^{2}e(h\frac{a}{q})
\end{align*}
where $w_{f,q_{0},q_{1}}$ will be defined in Remark \ref{Remark 3.5}. Note that $m_{1} \cup m_{2} = m.$

\noindent 
Let $g:\mathbb{N} \rightarrow \mathbb{R}$ be a function such that 
$\sum \frac{|g(n)|}{n^{s}}$ converges for $\Re(s)>1$, and let $\chi$ be a Dirichlet character. we denote 
\begin{align*}
D[g](s) &:= \sum_{n=1}^{\infty} \frac{g(n)}{n^{s}},\\
D[g](s,\chi) &:= \sum_{n=1}^{\infty} \frac{g(n)\chi(n)}{n^{s}},\\
D[g](s;q) &:= \sum_{n=1\atop (n,q)=1}^{\infty} \frac{g(n)}{n^{s}}\\
D[g](s,\chi,q_{0}) &:= \sum_{n=1}^{\infty} \frac{g(nq_{0})\chi(n)}{n^{s}},\\
D[g](s,q_{0};q_{1}) &:= \sum_{n=1 \atop (n,q_{1})=1}^{\infty}  \frac{g(nq_{0})}{n^{s}}
\end{align*}
for $\Re(s)>1$ and their analytic continuations elsewhere.
\begin{Def}(\cite{MRT1})
For $k>0$, 
we say that $F : \mathbb{N} \rightarrow \mathbb{C} $ is a $k$-divisor bounded function if
$F(n) \ll_{k}d_{2}(n)^{k}\log^{k}(2+n)$ for all $n \geq 1$.

\end{Def}
\begin{Rem}
The constant function $1$ is a 1-divisor bounded function. By the Deligne bound,  $\lambda_{f}(n)^{2}$ and $ \lambda_{f}(n^{2}) $ are 2-divisor bounded functions. 
\end{Rem}
\subsection{The circle method}

Applying the Hardy-Littlewood circle method and the Deligne bound, we see that 
$$\int_{I} |S_{f}(\alpha)|^{2}e(h\alpha)d\alpha = \sum_{X\leq n \leq 2X} \lambda_{f}(n)^{2}\lambda_{f}(n+h)^{2}+O\big(h\max_{2X-h \leq n \leq 2X} d_{2}(n)^{4}d_{2}(n+h)^{4}\big).$$
Since $d_{2}(n) \ll_{\epsilon} n^{\epsilon},$ the error term of the above equation is bounded by $hX^{\epsilon}.$
Now we state two propositions for major arcs and minor arcs estimates, and we will prove them in Section 3, Section 4-5 respectively.

\begin{Prop}\label{Prop1}(Major arc estimate) Let  $\epsilon>0$  be sufficiently small. Let $1\leq H \leq X^{1-\epsilon}.$ Let $Q=(\log X)^{B}$ for some sufficiently large $B>0.$ Then
\begin{equation}\int_{\mathcal{M}}|S_{f}(\alpha)|^{2}e(h\alpha)d\alpha-B_{h}X= O_{f,B,\epsilon}(X(\log X)^{-\frac{B}{2}}) \end{equation}
 for all but $O_{f,B,\epsilon}\big(H(\log X)^{1-\frac{B}{3}}\big)$ integers $h \in [1,H].$
\end{Prop}
\begin{Prop}\label{Propm}(Minor arc estimate) Let  $\epsilon>0$  be sufficiently small.  Let $X^{\frac{2}{3}+\epsilon} \leq H \leq X^{1-\epsilon},$ and let $k>0$ be fixed integer, $\rho =(\log X)^{-s}$ for some $s>0.$ Let $Q=(\log X)^{B}$ for some sufficiently large $B>0.$  Then 
 \begin{enumerate} \item $(m_{1}$-estimate$)$\label{Prop2} \begin{equation}\int_{m_{1} \cap [\theta -\frac{1}{2H}, \theta+\frac{1}{2H}]}
 |S_{f}(\alpha)|^{2} d\alpha \ll_{f,B,s,k,\epsilon}  \rho X (\log  X)^{-2k-1}\nonumber\end{equation} for all $\theta \in I.$
 \item $(m_{2}$-estimate$)$\label{Prop3}   \begin{equation}
 \int_{m_{2} \cap [\theta -\frac{1}{2H}, \theta+\frac{1}{2H}]}
 |S_{f}(\alpha)|^{2} d\alpha \ll_{f,B,s,k,\epsilon} \rho X (\log  X)^{-2k-1}\nonumber\end{equation} for all $\theta \in I.$
\end{enumerate}
\end{Prop}
\begin{Proof-non}
\noindent Set $k=A,B=20A$ and $s=3A.$  Since $I=\mathcal{M} \cup m,$
\begin{equation}
\begin{split}
\int_{I} |S_{f}(\alpha)|^{2}e(h\alpha)d\alpha = \int_{\mathcal{M}}  |S_{f}(\alpha)|^{2}e(h\alpha)d\alpha  +\int_{m}  |S_{f}(\alpha)|^{2}e(h\alpha)d\alpha.
\nonumber
\end{split}
\end{equation}
Therefore, we have
\begin{equation}\sum_{X\leq n \leq 2X} \lambda_{f}(n)^{2}\lambda_{f}(n+h)^{2}- \int_{\mathcal{M}}  |S_{f}(\alpha)|^{2}e(h\alpha)d\alpha  
 \end{equation}
\begin{equation}
=O_{\epsilon}\Big(\big|\int_{m}|S_{f}(\alpha)|^{2}e(h\alpha)d\alpha\big| +hX^{\epsilon^{2}} \Big).\nonumber
\end{equation}
    First, we show that for all but $O_{f,A,\epsilon}((\log X)^{-3A} H)$ many integers $h \in [1,H],$  

\begin{equation}
\int_{m} |S_{f}(\alpha)|^{2}e(h\alpha)d\alpha \ll_{f,A,\epsilon} X(\log X)^{-A}. \nonumber \end{equation}
By the Chebyshev inequality, it is enough to show that
$$\sum_{1 \leq h \leq H} \big|\int_{m}|S_{f}(\alpha)|^{2}e(h\alpha) d\alpha\big|^{2} \ll_{f,A,\epsilon} H X^{2}(\log X)^{-5A}.$$ The proof of this basically follows from \cite[Proposition 3.1]{MRT1}. 
Let $\Phi(x): \mathbb{R} \rightarrow \mathbb{R}$ be an even non-negative Schwartz function such that 

$$\Phi(x) \geq 1 \thinspace \textrm{\thinspace for \thinspace $x \in [0,1]$},$$
$$\widehat{\Phi}(s):= \int_{\mathbb{R}} \Phi(x)e(-xs)dx =0 \thinspace \thinspace \textrm{except \thinspace for \thinspace $s \in [-\frac{1}{2},\frac{1}{2}]$}.$$ 
By the Poisson summation formula, we have
$$\sum_{h} e(h(\alpha-\beta))\Phi(\frac{h}{H})= H \sum_{k} \widehat{\Phi}(H(\alpha-\beta-k)).$$
By \eqref{Fomenko}, we get
\begin{equation}
\begin{split}
\sum_{1 \leq h \leq H} \Big|\int_{m}|S_{f}(\alpha)|^{2}e(h\alpha) d\alpha \Big|^{2} &
\ll_{A} H\int_{m} |S_{f}(\alpha)|^{2} \int_{m \cap [\alpha-\frac{1}{2H}, \alpha+\frac{1}{2H}]} |S_{f}(\beta)|^{2}d\beta d\alpha \\ 
&\ll_{f,A} HX\log X \sup_{\alpha\in {m}}\int_{m\cap [\alpha-\frac{1}{2H}, \alpha+\frac{1}{2H}]} |S_{f}(\beta)|^{2}d\beta.
\\&\ll_{f,A} HX\log X \Big(\sup_{\alpha\in m}\int_{m_{1}\cap [\alpha-\frac{1}{2H}, \alpha+\frac{1}{2H}]} |S_{f}(\beta)|^{2}d\beta
\\&+ \sup_{\alpha\in m}\int_{m_{2}\cap [\alpha-\frac{1}{2H}, \alpha+\frac{1}{2H}]} |S_{f}(\beta)|^{2}d\beta\Big).
\nonumber
\end{split}
\end{equation}
By Proposition \ref{Propm}, this is bounded by 
$$ \ll_{f,A,\epsilon} HX^{2}(\log X)^{-5A}. $$
Therefore, for all but $O_{f,A,\epsilon}((\log X)^{-3A} H)$ many $h \in [1,H],$
$$\sum_{X\leq n \leq 2X} \lambda_{f}(n)^{2}\lambda_{f}(n+h)^{2}- \int_{\mathcal{M}}  |S_{f}(\alpha)|^{2}e(h\alpha)d\alpha  =O_{f,A,\epsilon}\big(X(\log X)^{-A}+hX^{\epsilon^{2}}\big).$$
By Proposition \ref{Prop1} and the condition $H \ll X^{1-\epsilon},$ we see that
$$\sum_{X\leq n \leq 2X} \lambda_{f}(n)^{2}\lambda_{f}(n+h)^{2}- B_{h}X=O_{f,A,\epsilon}\big(X(\log X)^{-A}+X (\log X)^{-10A}\big)$$
for all but $O_{f,A,\epsilon}\big(\max(H(\log X)^{1-\frac{20A}{3}}, (\log X)^{-3A} H)\big)$ many $h \in [1,H].$ \qed
\end{Proof-non}

\section{\rm{Major arc estimates}}
\noindent In this section, we will prove Proposition \ref{Prop1}. In Lemma \ref{Lemma 3.1} we transform the additive twist $e(\frac{a}{q}n)$ to multiplicative twists $\chi(n)$ for some Dirichlet characters $\chi.$ Then we apply some analytic properties of L-functions twisted by Dirichlet characters to get the average values of $\lambda_{f}(q_{0}n)\chi(n)$ in Lemma \ref{Lemma 3.4}, Lemma \ref{Lemma 3.7}.
After attaching $e(n\beta)$ to $\lambda_{f}(n)^{2}e(\frac{a}{q}n)$, we get
$$\sum_{X\leq n\leq 2X} \lambda_{f}(n)^{2}e(\frac{a}{q} n)e(\beta n) \sim \int_{X}^{2X}\sum_{q=q_{0}q_{1}}\frac{\mu(q_{1})}{\phi(q_{1})q_{0}} w_{f,q_{0},q_{1}}e(\beta x) dx$$ (see Lemma \ref{Lemma 3.8}), and this is beneficial to apply the Fourier inversion formula (see \eqref{3FIF}).
\begin{Lemma}\label{Lemma 3.1}Let $s=\sigma +it$ where $\sigma, t \in \mathbb{R}, \sigma>1.$ Let $(a,q)=1.$ Then

$$\sum_{n=1}^{\infty} \lambda_{f}(n)^{2}e(\frac{a}{q}n)n^{-s}= \sum_{q=q_{0}q_{1}} \frac{1}{\phi(q_{1})q_{0}^{s}
} \sum_{\chi (\mathrm{mod} \thinspace q_{1})} \tau(\bar{\chi})\chi(a) \sum_{n=1 \atop (n,q_{1})=1}^{\infty} \lambda_{f}(q_{0}n)^{2}\chi(n)n^{-s}$$
where $$\tau(\chi):= \sum_{1 \leq m \leq q_{1}, \atop(m,q_{1})=1} \chi(m)e(\frac{m}{q_{1}})$$ is the Gauss sum. 
\end{Lemma}
\begin{proof}
see \cite[Lemma 2.9]{MRT1}.
\end{proof}\noindent
From now on we regard $q=q_{0}q_{1}.$ 

\begin{Lemma}$($\cite[Lemma 3]{RA2}$)$
Let $s=\sigma+it$ for $\sigma, t \in \mathbb{R},$ and let 
$$A_{1}(s)=\sum_{n=1}^{\infty} \lambda_{f}(n)^{2}n^{-s},$$
$$\lambda_{f}(p)=2\cos\theta_{p}.$$
Then for $\sigma > 1$, $A_{1}(s)$ has the absolutely convergent Euler product such that 
$$A_{1}(s)=\frac{\zeta^{2}(s)\psi(s)}{\zeta(2s)}$$
where $$\psi(s)=\prod_{p}(1-2p^{-s}\cos2\theta_{p}+p^{-2s})^{-1}.$$
The function $\zeta(s)\psi(s)$ is holomorphic and non-zero for $\sigma \geq 1$.
\end{Lemma}

\begin{Rem}\label{Remark 3.3}
Let $q_{0}=\prod_{i=1}^{k} p_{i}^{\alpha_{i}}$ for some primes $p_{i}$. Let $$D(s):= \sum_{n=1 \atop (n,q_{1})=1}^{\infty} \frac{\lambda_{f}(q_{0}n)^{2}}{n^{s}}.$$
 Then
$$D(s)=\frac{\zeta^{2}(s)\psi(s)}{\zeta(2s)} \times R_{q_{0},q_{1}}(s)$$
 where
$$R_{q_{0},q_{1}}(s):= \prod_{p|q} \frac{(1-p^{-s})^{2}}{1-p^{-2s}} (1-2p^{-s}\cos2\theta_{p}+p^{-2s}) \prod_{i=1}^{k} (\lambda_{f}(p_{i}^{\alpha_{i}})^{2}+\lambda_{f}(p_{i}^{\alpha_{i}+1})^{2}p_{i}^{-s}+...).$$
So that $D(s)$ satisfies the same convexity bound of  $R_{q_{0},q_{1}}(s)L(f\otimes f,s).$  The convexity bound of $L(f\otimes f,s)$ on $\Re(s) =\frac{1}{2}$ means $L(f\otimes f,s) \ll_{f,\epsilon} |s|^{1+\epsilon}$.
\end{Rem}
\noindent By Remark \ref{Remark 3.3}, we get the following asymptotic formula.

\begin{Lemma}\label{Lemma 3.4} Let  $\epsilon>0$  be sufficiently small. Then

$$ \sum_{X \leq n\leq 2X\atop(n,q_{1})=1} \lambda_{f}(q_{0}n)^{2} =\nu_{f,q_{0},q_{1}}X+O_{q_{0},\epsilon,f}(X^{\frac{3}{4}+\epsilon})$$
for some $\nu_{f,q_{0},q_{1}}$.
\end{Lemma}

\begin{proof}
Let $0<Y<\frac{X}{5}.$
Let $\eta$ be a fixed smooth function with a compact support $[X-Y,2X+Y]$ such that $\eta(\frac{x}{X})=1$ for $ X\leq n\leq 2X$, $\eta(x) \in [0,1]$. By the Deligne bound, we have
$$ \sum_{X \leq n \leq 2X \atop (n,q_{1})=1}\lambda_{f}(q_{0}n)^{2}= \sum_{n=1 \atop (n,q_{1})=1}^{\infty}\lambda_{f}(q_{0}n)^{2}\eta(\frac{n}{X}) + O_{q_{0},\epsilon,f}( YX^{\epsilon}).$$
Let $$\tilde{\eta}(s):= \int_{0}^{\infty} \eta(x)x^{s-1}dx.$$
By the Mellin inversion formula, we get
$$\sum_{X\leq n\leq 2X   \atop(n,q_{1})=1} \lambda_{f}(q_{0}n)^{2} =\frac{1}{2\pi i} \int_{2-i\infty}^{2+i\infty} D(s) \tilde{\eta}(s)X^{s} ds+O_{q_{0},\epsilon,f}(YX^{\epsilon}).$$
Let $\nu_{f,q_{0},q_{1}}= \textrm{Res}_{s=1}D(s) \times \tilde{\eta}(1)$. It is easy to check that $\tilde{\eta}(s) \ll_{k} (\frac{X}{Y})^{k-1} |s|^{-k}$ for any $k \in \mathbb{N}.$
By changing the contour to the line $\Re(s)=\frac{1}{2},$ we have
$$\sum_{X\leq n\leq 2X\atop(n,q_{1})=1} \lambda_{f}(q_{0}n)^{2}=\nu_{f,q_{0},q_{1}}X+ \int_{\frac{1}{2}-i\infty}^{\frac{1}{2}+i\infty} D(s)\tilde{\eta}(s)X^{s}ds + O_{q_{0},\epsilon,f}(YX^{\epsilon}).$$
By the convexity bound, $D(s) \ll_{q_{0},\epsilon,f} |s|^{1+\epsilon}$  on $\Re(s)=\frac{1}{2}.$ Therefore, we see that
\begin{equation}\begin{split}\int_{\frac{1}{2}-i\infty}^{\frac{1}{2}+i\infty} D(s)\tilde{\eta}(s)X^{s}ds
&\ll_{q_{0},\epsilon,f} \Big|\int_{\frac{1}{2}-i\frac{X^{1+\epsilon}}{Y}}^{\frac{1}{2}+i\frac{X^{1+\epsilon}}{Y}} D(s)\tilde{\eta}(s)X^{s}ds\Big| + 1\\ &=O_{q_{0},\epsilon,f}\Big(X^{\epsilon}(\frac{X^{\frac{3}{2}}}{Y})\Big).\end{split} \nonumber \end{equation} 
When $Y=X^{\frac{3}{4}},$ the big-O term of the above inequality is  $O_{q_{0},\epsilon,f}(X^{\frac{3}{4}+\epsilon}).$

\end{proof}

\begin{Rem}\label{Remark 3.5}
By the conditions on $\eta$, we have $\tilde{\eta}(1)=1+O(X^{-\frac{1}{4}}).$
From Remark \ref{Remark 3.3} and Lemma \ref{Lemma 3.4}, we get
 \begin{equation}\begin{split}\sum_{X\leq n\leq 2X \atop(n,q_{1})=1} \lambda_{f}(q_{0}n)^{2}- w_{f,q_{0},q_{1}}X&=O_{\epsilon}\Big(X^{\frac{3}{4}+\epsilon}+|\textrm{Res}_{s=1}D(s)||\tilde{\eta}(1)-1|X\Big)\\&=O_{\epsilon}(|\textrm{Res}_{s=1}D(s)|X^{\frac{3}{4}+\epsilon})\end{split} \nonumber \end{equation}  where $$w_{f,q_{0},q_{1}} = c_{1,f}\prod_{p|q} (\frac{p-1}{p+1})(1-\frac{\lambda_{f}(p)^{2}-2}{p}+\frac{1}{p^{2}}) \prod_{p^{l} \parallel q_{0}}(\lambda_{f}(p^{l})^{2}+\lambda_{f}(p^{l+1})^{2}p^{-1}+...).$$ By the Deligne bound, when $q=q_{0}q_{1} \geq 2,$ we see that 
 \begin{equation}\label{3wfq}\begin{split} |w_{f,q_{0},q_{1}}| &\leq c_{1,f} \prod_{p|q} (\frac{p-1}{p+1})(1+\frac{2}{p}+\frac{1}{p^{2}})\prod_{p^{l} \parallel q_{0}}\big((l+1)^{2}+(l+2)^{2}p^{-1}+...)
 \\& \ll  c_{1,f}(\log_{2} q)^{2} \prod_{p|q} (\frac{p-1}{p+1})(1+\frac{2}{p}+\frac{1}{p^{2}})\prod_{p^{l} \parallel q_{0},\atop p \neq 2}(l+1)^{2}\frac{1}{1-\frac{9}{4p}} 
 \\& \ll  c_{1,f}(\log_{2} q)^{4} \prod_{p|q} (1-\frac{1}{p^{2}})\prod_{p|q_{0} \atop p \neq 2 }\frac{1}{1-\frac{9}{4p}} 
 \\ &\ll c_{1,f} (\log_{2}q)^{7}.
 \end{split} \end{equation}
 Since $w_{f,q_{0},q_{1}} \sim \textrm{Res}_{s=1}D(s),$ we conclude that 
  \begin{equation}\label{3squareerror}\sum_{X\leq n\leq 2X \atop(n,q_{1})=1} \lambda_{f}(q_{0}n)^{2}- w_{f,q_{0},q_{1}}X=O_{\epsilon,f}\big((\log q)^{7}X^{\frac{3}{4}+\epsilon} \big). \end{equation}
\end{Rem}
\begin{Rem}\label{Remark 3.6}
Recently, Bingrong Huang \cite{huang2021rankinselberg} proved that
\begin{equation}\label{BH}\sum_{1 \leq n\leq X} \lambda_{f}(n)^{2} = c_{1,f}X+ O_{f}(X^{\frac{3}{5}-\frac{1}{560}+o(1)}).\end{equation}
Following his methods and notations, one can prove that
$$\sum_{1 \leq n \leq X \atop (n,q_{1})=1} \lambda_{f \times f}(q_{0}n)= w_{f,q_{0},q_{1}}X+O_{f}(\sup_{t\in \mathbb{R}}|R_{q_{0},q_{1}}(\frac{1}{2}+it)|X^{\frac{3}{5}-\frac{1}{560}+o(1)}).$$
By the Deligne bound and the following bounds 
$$|(1-\frac{2}{p^{\frac{1}{2}}-1})(1+\frac{1}{p^{\frac{1}{2}}})^{2}|\leq 11.156... \leq 12 \thinspace\;\; \textrm{for} \;\;p \geq 2,\;\;\frac{1}{1-\frac{9}{4p^{\frac{1}{2}}}} \leq 6.685... \leq 7 \thinspace \;\; \textrm{for} \;\;p \geq 7,$$ we see that 
\begin{equation}\label{3Rq0q1} \begin{split} R_{q_{0},q_{1}}(\frac{1}{2}+it) &\ll (\log_{2} q)^{6} |\prod_{p|q}(1-\frac{2}{p^{\frac{1}{2}}-1})(1+\frac{1}{p^{\frac{1}{2}}})^{2}|\prod_{p_{i}^{\alpha_{i}}\parallel q_{0} \atop p_{i} \neq 2,3,5}\frac{\alpha_{i}^{2}}{1-\frac{9}{4p_{i}^{\frac{1}{2}}}}\\&\ll 84^{\log_{2} q} (\log_{2} q)^{6}\\& \ll q^{7} \end{split}\end{equation} $(\log_{2}(84)= 6.39...).$
Therefore, we have
\begin{equation}\label{3q0square}\begin{split}\sum_{1 \leq n \leq X \atop (n,q_{1})=1} \lambda_{f}(q_{0}n)^{2}&=\sum_{1 \leq l \leq X^{\frac{1}{2}}} \mu(l) \Big(w_{f,q_{0},q_{1}}\frac{X}{l^{2}}+O_{f}(q^{7}(\frac{X}{l^{2}})^{\frac{3}{5}-\frac{1}{560}+o(1)})\Big)\\&=w_{f,q_{0},q_{1}}X+O_{f}(q^{7}X^{\frac{3}{5}-\frac{1}{560}+o(1)}).\end{split}\end{equation}
This bound is significantly smaller than the error term in Lemma \ref{Lemma 3.4}. This power saving does not make any difference for our results, so we apply Lemma \ref{Lemma 3.4} instead of \eqref{BH}. One may also apply \eqref{BH} for the following Lemma. 
\end{Rem}

\begin{Lemma}\label{Lemma 3.7} Let  $\epsilon>0$  be sufficiently small. Then for any non-principal character $\chi$ modulo $q_{1},$ 
$$ \sum_{X \leq n\leq 2X \atop(n,q_{1})=1} \chi(n)\lambda_{f}(q_{0}n)^{2} \ll_{\epsilon,f} q^{4+\frac{\epsilon}{2}}X^{\frac{3}{4}+\epsilon}.$$
\end{Lemma}
\begin{proof}
Let $0<Y<\frac{X}{5}.$
Let $\eta$ be a fixed smooth function with a compact support $[X-Y,2X+Y]$ such that $\eta(\frac{x}{X})=1$ for $ X\leq n\leq 2X$, $\eta(x) \in [0,1]$. Then
$$ \sum_{X\leq n\leq 2X\atop(n,q_{1})=1} \chi(n)\lambda_{f}(q_{0}n)^{2} -\sum_{X\leq n\leq2X+Y  \atop(n,q_{1})=1} \chi (n)\lambda_{f}(q_{0}n)^{2}\eta(\frac{n}{X}) = O_{\epsilon,f}(YX^{\epsilon}).$$
Let $$D(s,\chi)=\sum_{n=1}^{\infty} \frac{\chi(n)\lambda_{f}(q_{0}n)^{2}}{n^{s}}.$$
 By the Mellin inversion formula, we get 
\begin{equation}\label{3chisquare}\sum_{X\leq n\leq2X \atop(n,q_{1})=1} \chi (n)\lambda_{f}(q_{0}n)^{2} =\frac{1}{2\pi i} \int_{2-i\infty}^{2+i\infty} D(s,\chi) \tilde{\eta}(s)X^{s} ds +O_{\epsilon,f}(YX^{\epsilon}).\end{equation}
Since $D(s,\chi)$ has no pole at $s=1$ (see  \cite[Chapter 7.3]{Go}. Although the results there are stated for Maass cusp forms, the same proof applies equally well to the holomorphic cusp forms),  by changing the contour of the integral \eqref{3chisquare} to the line $\Re(s)=\frac{1}{2},$ we have
\begin{equation}\frac{1}{2\pi i} \int_{2-i\infty}^{2+i\infty} D(s,\chi) \tilde{\eta}(s)X^{s} ds =\frac{1}{2\pi i} \int_{\frac{1}{2}-i\infty}^{\frac{1}{2}+i\infty} D(s,\chi) \tilde{\eta}(s)X^{s}ds+ O_{f,k}(1) \nonumber \end{equation} by using the fact that $\tilde{\eta}(s) \ll_{k} (\frac{X}{Y})^{k-1} |s|^{-k}$ for any $k \in \mathbb{N}.$
By the convexity bound and \eqref{3Rq0q1}, $D(s,\chi) \ll_{\epsilon,f} q^{1+\epsilon+7}|s|^{1+\epsilon}$ on $\Re(s)=\frac{1}{2}.$ Therefore, we see that  
\begin{equation}\label{3error1}\begin{split}\int_{\frac{1}{2}-i\infty}^{\frac{1}{2}+i\infty} D(s,\chi)\tilde{\eta}(s)X^{s}ds
&\ll_{\epsilon,f} \Big|\int_{\frac{1}{2}-i\frac{X^{1+\epsilon}}{Y}}^{\frac{1}{2}+i\frac{X^{1+\epsilon}}{Y}} D(s,\chi)\tilde{\eta}(s)X^{s}ds \Big| + O_{\epsilon,f}(1)\\ &=O_{\epsilon,f}\Big(q^{1+\epsilon+7}X^{\epsilon}(\frac{X^{\frac{3}{2}}}{Y})\Big).\end{split} \end{equation} 
When $Y=q^{4+\frac{\epsilon}{2}}X^{\frac{3}{4}},$ the big-O term in \eqref{3error1} is 
$O_{\epsilon,f}\Big(q^{4+\frac{\epsilon}{2}}X^{\frac{3}{4}+\epsilon}\Big).$ Therefore, \eqref{3chisquare} is bounded by $q^{4+\frac{\epsilon}{2}}X^{\frac{3}{4}+\epsilon}.$
\end{proof}
\noindent Now we get the asymptotic estimates $$\sum_{X \leq n \leq 2X} \lambda_{f}(n)^{2}e\big((\frac{a}{q}+\beta)n\big) \; \sim \int_{X}^{2X}\sum_{q=q_{0}q_{1}}\frac{\mu(q_{1})}{\phi(q_{1})q_{0}} w_{f,q_{0},q_{1}}e(\beta x) dx$$ with some admissible error terms.

\begin{Lemma}\label{Lemma 3.8} $($\cite{MRT1}, \rm{Lemma 2.1}$)$
Let $f:[X,2X] \rightarrow \mathbb{C}$ be a smooth function. Then for any function $g:\mathbb{N} \rightarrow \mathbb{C}$ and absolutely integrable function $h:[X,2X] \rightarrow \mathbb{C}$, 

$$\sum_{X\leq n \leq 2X} f(n)g(n) - \int_{X}^{2X} f(x)h(x) dx \leq |f(2X)E(2X)|+\int_{X}^{2X} |f'(x)E(x)|dx$$ 
where $$E(x):=\max_{X\leq X' \leq x}\big|\sum_{X\leq n \leq X'}g(n)-\int_{X}^{X'} h(x)dx\big|.$$
\end{Lemma}
\begin{proof}
This is easily deduced from summation by parts. 
\end{proof}
\begin{Rem}\label{Remark 3.9}
In our case, we take $$f(n)=e(n\beta),\; g(n)=\lambda_{f}(n)^{2}e(\frac{a}{q}n),\; h(x)=h=\sum_{q=q_{0}q_{1}}\frac{\mu(q_{1})}{\phi(q_{1})q_{0}} w_{f,q_{0},q_{1}}.$$ Notice that $h(x)$ is a constant on $x.$ It is known that $|\tau(q)| \ll q^{\frac{1}{2}}.$  
Since $1 \leq q \leq Q,$ \begin{equation}\begin{split} E(x)&=\max_{X \leq X' \leq x}\big|\sum_{q=q_{0}q_{1}}\frac{\mu(q_{1})}{\phi(q_{1})q_{0}} w_{f,q_{0},q_{1}}X'+O_{f,\epsilon}(q^{5+\frac{\epsilon}{2}}X^{\frac{3}{4}+\epsilon})-X'h\big| \\&= O_{f,\epsilon}(X^{\frac{3}{4}+\epsilon}Q^{5+\frac{\epsilon}{2}}) \end{split} \end{equation}  for 
$x \in [X,2X].$
\end{Rem}
\noindent Now we are ready to prove Proposition \ref{Prop1}.
\begin{Proof-non4} 
Let $\alpha=\frac{a}{q}+\beta$ for some $ q \leq Q, (a,q)=1.$ 
By Lemma \ref{Lemma 3.4},  Remark \ref{Remark 3.5} and Remark \ref{Remark 3.9}, we see that
\begin{equation}\begin{split}S_{f}(\frac{a}{q}+\beta)&=\sum_{X\leq n\leq 2X} \lambda_{f}(n)^{2}e(\frac{a}{q} n)e(\beta n)\\&=\int_{X}^{2X}\sum_{q=q_{0}q_{1}}\frac{\mu(q_{1})}{\phi(q_{1})q_{0}} w_{f,q_{0},q_{1}}e(\beta x) dx +O_{\epsilon}\big((|\beta|+\frac{1}{X})X^{\frac{7}{4}+\frac{\epsilon}{8}}Q^{5+\frac{\epsilon}{2}}\big). \end{split}\nonumber \end{equation} 
Let $$D_{q}:= \sum_{q=q_{0}q_{1}} \frac{\mu(q_{1})}{\phi(q_{1})q_{0}}w_{f,q_{0},q_{1}}.$$ 
By \eqref{3wfq}, we have
\begin{equation}\label{3dq} \begin{split} 
|D_{q}| &\ll_{f} \sum_{q=q_{0}q_{1}} \frac{1}{\phi(q_{1})q_{0}}|w_{f,q_{0},q_{1}}| \\&\ll_{f} \sum_{q=q_{0}q_{1}} \frac{1}{q} c_{1,f}(\log_{2} q)^{7} \\&\ll_{f} \frac{(\log q)^{7}}{q} \sum_{q=q_{0}q_{1}}1 \\&\ll_{f} \frac{d_{2}(q) (\log q)^{7}}{q}.
\end{split} \end{equation}
By the Fourier inversion formula, we have
\begin{equation}\label{3FIF} \begin{split}\int_{\mathbb{R}}\big|\int_{X}^{2X} e(\beta x) dx\big|^{2} e(\beta h) d\beta &= \int_{\mathbb{R}} 1_{[X,2X]}(x)1_{[X,2X]}(x+h) dx\\&=X+O(h)\\&=\big(1+O(X^{-\epsilon})\big)X.\nonumber\end{split}\end{equation} 
By the Deligne bound, we have
 $$\int _{|\beta|\leq X^{-\frac{5}{6}-2\epsilon }}|S_{f}(\frac{a}{q}+\beta)|^{2}e\big(h(\frac{a}{q}+\beta)\big)d\beta=e(h\frac{a}{q})X |D_{q}|^{2}\big(1+O_{\epsilon}(X^{-\epsilon})\big) $$ 
\begin{equation}\begin{split}
& + O_{f,\epsilon}\Big(\int_{|\beta|\leq X^{-\frac{5}{6}-2\epsilon}}\big(|D_{q}|X^{\frac{7}{4}+\frac{\epsilon}{8}}Q^{5+\frac{\epsilon}{2}}+  (|\beta|^{2}+\frac{1}{X^{2}})X^{\frac{7}{2}+\frac{\epsilon}{4}}Q^{10+\epsilon} \big)d\beta \\&+ \int_{|\beta|>X^{-\frac{5}{6}-2\epsilon}}     \frac{|D_{q}|^{2}}{ |\beta|^{2}}d\beta\Big) 
\\&= e(h\frac{a}{q}) |D_{q}|^{2}X +  O_{f,\epsilon}\big(|D_{q}|^{2}X^{1-\epsilon}+|D_{q}|X^{\frac{11}{12}+\epsilon}+|D_{q}|^{2} X^{\frac{5}{6}+2\epsilon}\big). \nonumber \end{split}\end{equation} 
Therefore, we see that
\begin{equation}
\begin{split}
\sum_{1 \leq q \leq Q} \sum_{1 \leq a \leq q \atop(a,q)=1}& \int _{|\beta|\leq X^{-\frac{5}{6}-2\epsilon} }|S_{f}(\frac{a}{q}+\beta)|^{2}e\big(h(\frac{a}{q}+\beta)\big)d\beta 
\\&=X\sum_{1 \leq q\leq Q} c_{q}(h)|D_{q}|^{2}+O_{f,\epsilon}(X^{1-\epsilon})\\
&=X\big(B_{h} - \sum_{q> Q} c_{q}(h)|D_{q}|^{2}\big)+O_{f,\epsilon}(X^{1-\epsilon})
\nonumber
\end{split}
\end{equation}
where $c_{q}(h)$ is the Ramanujan sum. It is known that for fixed $h,$ $|c_{q}(h)|\leq (q,h).$  Since $|D_{q}|^{2}\ll_{f.\epsilon} \frac{1}{q^{2-\epsilon}},$ we have

\begin{equation} \begin{split} \sum_{q> Q} c_{q}(h)|D_{q}|^{2} &\ll_{f,\epsilon} \sum_{d|h} \sum_{q>Q \atop (h,q)=d} \frac{d}{q^{2-\epsilon}}\\ &\ll_{f,\epsilon} \sum_{d|h} \sum_{q>\frac{Q}{d}} \frac{d}{(dq)^{2-\epsilon}}\\ &\ll_{f,\epsilon}  \sum_{d|h}Q^{-1+\epsilon}    .\end{split}\end{equation}
By the Markov inequality and using the fact that $$\sum_{1 \leq n\leq H} d_{2}(n) \sim H \log H,$$ we see that  \begin{equation}\label{3divisor}d_{2}(h) \ll (\log X)^{\frac{B}{2.9}}\end{equation} for all but $O_{f,B,\epsilon}\big(H(\log X)^{1-\frac{B}{3}}\big)$ many values of $h \in [1,H].$
 Therefore, by \eqref{3divisor}, we conclude that 
\begin{equation}\int_{\mathcal{M}} |S_{f}(\beta)|^{2}e(h\beta)d\beta -B_{h}X=O_{f,B,\epsilon}(X(\log X)^{-\frac{B}{2}}) \nonumber \end{equation}
 for all but $O_{f,B,\epsilon}\big(H(\log X)^{1-\frac{B}{3}} \big)$ many values of $ h \in [1,H].$ \qed
\end{Proof-non4}
\section{$m_{2}$-estimate }
\noindent In this section we will prove the $m_{2}$-estimate in Proposition \ref{Propm}. In subsection 4.1 we will apply Lemma \ref{Lemma 1.3} to treat $\lambda_{f}(n^{2})$ in $S_{f}(\alpha)$ for $n> HX^{-\epsilon^{2}}.$ In subsection 4.2 we will basically follow the methods applied in \cite{MRT1}  (see Lemma \ref{Lemma 4.5}). In subsection 4.3, subsection 4.4 we will finish proving the $m_{2}$-estimate in Proposition \ref{Propm}.
\subsection{A modification of $\lambda_{f}(n)^{2}$}
Subsection 4.1 is devoted to show that 
 \begin{equation}\label{41}\int_{m_{2} \cap [\theta -\frac{1}{2H}, \theta+\frac{1}{2H}]}
 |S_{f}(\alpha)|^{2} d\alpha  \ll_{f,B,s,k,\epsilon} \rho X (\log X)^{-2k-1},\end{equation}
 where $\rho, s$ and $B,k$ are stated in Proposition \ref{Propm}.
\noindent First, we apply the Hecke relations to separate the above integral.
\begin{Lemma}\label{Lemma 4.1} Let $1\leq Y \leq 2X.$ Then 
\begin{equation}\begin{split}\int_{m_{2} \cap [\theta -\frac{1}{2H}, \theta+\frac{1}{2H}]}
 &|S_{f}(\alpha)|^{2} d\alpha \ll\int_{m_{2}\cap [\theta -\frac{1}{2H}, \theta+\frac{1}{2H}]} \Big(  \Big|\sum_{Y\leq d \leq 2X} \lambda_{f}(d^{2}) \sum_{\frac{X}{d} \leq k \leq \frac{2X}{d}}e(dk \alpha)\Big|^{2} \\&
+\Big|\sum_{1 \leq d \leq Y} \lambda_{f}(d^{2}) \sum_{\frac{X}{d} \leq k \leq \frac{2X}{d}}e(dk \alpha)\Big|^{2}\Big) d\alpha.
\end{split} \nonumber
\end{equation}
\end{Lemma}
\begin{proof}

By the Hecke relations, we have $$ \lambda_{f}(n)^{2} = \sum_{d|n} \lambda_{f}(d^{2}).$$
Therefore, we see that
\begin{equation} 
\begin{split}S_{f}(\alpha) &= \sum_{ X \leq n \leq 2X} \lambda_{f}(n)^{2}e(n\alpha) \\&=\sum_{ X \leq n \leq 2X} \big(\sum_{d|n} \lambda_{f}(d^{2})\big)e(n\alpha) \\&= \sum_{1\leq d \leq 2X} \lambda_{f}(d^{2}) \sum_{\frac{X}{d} \leq k \leq \frac{2X}{d}} e(dk\alpha). 
\end{split}\nonumber \end{equation} By separating the sum over $d,$ we have
\begin{equation} \begin{split}S_{f}(\alpha)= \sum_{X \leq n \leq 2X} \lambda_{f}(n)^{2}e(n\alpha) \thinspace &- \sum_{Y \leq d \leq 2X} \lambda_{f}(d^{2}) \sum_{\frac{X}{d} \leq k \leq \frac{2X}{d}}e(dk \alpha) \thinspace \\&+
\sum_{Y\leq d \leq 2X} \lambda_{f}(d^{2}) \sum_{\frac{X}{d} \leq k \leq \frac{2X}{d}}e(dk \alpha).\end{split} \nonumber\end{equation}
\end{proof}
\noindent We now bound the first integral at the right-hand side of Lemma \ref{Lemma 4.1}.
\begin{Lemma}\label{Lemma 4.2} Let $1\leq Y \leq 2X.$ Then
\begin{equation}\label{4error1}\int_{m_{2}\cap [\theta -\frac{1}{2H}, \theta+\frac{1}{2H}]}\Big|\sum_{Y \leq d \leq 2X} \lambda_{f}(d^{2}) \sum_{\frac{X}{d} \leq k \leq \frac{2X}{d}}e(dk \alpha)\Big|^{2} d\alpha \ll_{f,B,\epsilon} X^{2}Y^{-\frac{1}{2}+2\epsilon}H^{-1}.\end{equation}
\end{Lemma}
\begin{proof}By rearranging the sum and Lemma 1.3, we see that 
\begin{equation}\label{4pointwise}
\begin{split}
\sum_{Y \leq d \leq 2X} \lambda_{f}(d^{2}) \sum_{\frac{X}{d} \leq k \leq \frac{2X}{d}}e(dk \alpha) &\ll   \sum_{1\leq k\leq \frac{2X}{Y}} \sum_{Y \leq d \leq \frac{2X}{k}} \lambda_{f}(d^{2})e(dk\alpha) \\& \ll_{f,\epsilon} \sum_{1 \leq k \leq \frac{2X}{Y}} (\frac{2X}{k})^{\frac{3}{4}+\epsilon} \\& \ll_{f,\epsilon} X^{\frac{3}{4}+\epsilon} \sum_{1\leq k \leq \frac{2X}{Y}} k^{-\frac{3}{4}-\epsilon} \\& \ll_{f,\epsilon} \frac{X}{Y^{\frac{1}{4}-\epsilon}}. 
\end{split} 
\end{equation}
Using \eqref{4pointwise} as a pointwise bound, we see that
\begin{equation}\int_{m_{2}\cap [\theta -\frac{1}{2H}, \theta+\frac{1}{2H}]}\Big|\sum_{Y \leq d \leq 2X} \lambda_{f}(d^{2}) \sum_{\frac{X}{d} \leq k \leq \frac{2X}{d}}e(dk \alpha)\Big|^{2} d\alpha \ll_{f,B,\epsilon} X^{2}Y^{-\frac{1}{2}+2\epsilon}H^{-1}.\end{equation}
\end{proof}
\noindent Since $H>X^{\frac{2}{3}+\epsilon},$ when $Y=HX^{-\epsilon^{2}},$ the right-hand side of the above inequality is bounded by $X^{1-\epsilon}.$ 
Therefore, to prove \eqref{41}, it is enough to show that 
\begin{equation} \begin{split}
\int_{m_{2} \cap [\theta -\frac{1}{2H}, \theta+\frac{1}{2H}]}  \Big|\sum_{1 \leq d \leq HX^{-\epsilon^{2}}} &\lambda_{f}(d^{2}) \sum_{\frac{X}{d} \leq k \leq \frac{2X}{d}}e(dk \alpha)\Big|^{2} d\theta \\&\ll_{f,B,s,k,\epsilon} \rho X (\log X)^{-2k-1}.
\end{split}\nonumber
\end{equation}
\subsection{Combinatorial decomposition}
Let $$b(n):= \sum_{m|n \atop 1\leq m < HX^{-\epsilon^{2}} }\lambda_{f}(m^{2}),$$
 $$S_{b}(\alpha) := \sum_{X \leq n \leq 2X} b(n)e(n\alpha).$$ 
By the Deligne bound, we see that  $$ \thinspace b(n) \leq \sum_{m|n} d_{2}(m)^{2} \leq d_{2}(n)^{3}, $$ which means that $b(n)$ is a 3-divisor bounded function. 
We will use the following lemma on $S_{b}(\alpha).$ By Lemma \ref{Lemma 4.3} , verifying \eqref{41} is reduced to verifying some integral bounds of Dirichlet polynomials.
\begin{Lemma}\label{Lemma 4.3}
Let $(a,q)=1,$ and let $\gamma, \rho'$ be real numbers such that $|\gamma| \ll \rho' \ll 1.$ Let 
\begin{equation}
I_{\gamma,\rho'}:=  \{ t \in \mathbb{R}: \rho'|\gamma|X \leq |t| \leq \frac{|\gamma|X}{\rho'} \}.
\end{equation}
Then 
$$\int_{\gamma-\frac{1}{H}}^{\gamma+\frac{1}{H}}\big|S_{b}(\frac{a}{q}+\theta)\big|^{2} d \theta$$
 \begin{equation}\label{longintegral}\ll \frac{d_{2}(q)^{4}}{q|\gamma|^{2}H^{2}} \sup_{q=q_{0}q_{1}} \int_{ I_{\gamma,\rho'}} \Big(\sum_{\chi(\mathrm{mod} \thinspace q_{1})}\int_{t-|\gamma|H}^{t+|\gamma|H} \big|\sum_{\frac{X}{q_{0}} \leq n \leq \frac{2X}{q_{0}}} \frac{b(q_{0}n)\chi(n)}{n^{\frac{1}{2}+it'}}\big| dt'\Big)^{2}dt\end{equation} $$ +(\rho'+\frac{1}{|\gamma| H})^{2} \int_{\mathbb{R}}\big(H^{-1} \sum_{x \leq n \leq x+H} |b(n)|\big)^{2} dx.$$

\end{Lemma}
\begin{proof} See \cite[Corollary 5.3]{MRT1}. $b(n)$ can be replaced by any function. \end{proof}
\noindent In order to apply Lemma \ref{Lemma 4.3}  to the intervals in $m_{2},$ we set $$\gamma \in [\frac{1}{2}X^{-\frac{1}{2}-\epsilon},\;\frac{1}{qQ}], \;\rho'=Q^{-\frac{1}{2}}.$$

\begin{Lemma} $( \thinspace $\cite[Proposition 6,1]{MRT1}, \rm{Combinatorial decomposition}$)$
Let $\epsilon>0$ be sufficiently small, and let $B>0$ be sufficiently large.
Let $X\geq 2$, and set $H>X^{\frac{2}{3}+3\epsilon}.$ 
Set $Q:=(\log X)^{B}$, $\rho'=Q^{-\frac{1}{2}}$, and let $1\leq q_{1}\leq Q.$ 
Let $\gamma'$ be a quantity such that $\frac{1}{2}X^{-\frac{1}{2}-\epsilon}\leq \gamma' \leq \frac{1}{q_{1}Q}$.
Let $g:\mathbb{N}\rightarrow \mathbb{C}$ be a function which is one of the  following forms;
\begin{description} 
\item[\normalfont{Type} $\rm{d_{1}}$ \normalfont{sum} ]
\begin{equation}\label{Typed1} g= \varrho \ast \varsigma \end{equation}
for some 3-divisor bounded arithmetic function $\varrho: \mathbb{N} \rightarrow \mathbb{C}\thinspace$  where $\varrho$ is supported on $[N,2N]$, $\varsigma=1_{(M,2M]}$ for some $N,M$ obeying the bounds
$$ 1 \ll N \ll_{\epsilon} X^{\epsilon^{2}}, \; \frac{X}{Q} \ll NM \ll_{\epsilon} X,$$ and $$X^{-\epsilon^{2}}H \ll M \ll X.$$ 
\item[ \normalfont{Type} $\rm{\RomanNumeralCaps{2}}$ \normalfont{sum} ] 
\begin{equation} g=\varrho \ast \varsigma \end{equation}
for some 3-divisor bounded arithmetic functions $\varrho, \varsigma: \mathbb{N} \rightarrow \mathbb{C}\thinspace$ where $\varrho$ is supported on $[N,2N]$, $\varsigma$ is supported on $[M,2M]$ for some $N,M$ satisfying the bounds
$$ X^{\epsilon^{2}} \ll N \ll X^{-\epsilon^{2}}H, \; \frac{X}{Q} \ll NM \ll_{\epsilon} X.$$
\end{description}
Then 
\begin{equation}\label{4cd} \int_{I_{\gamma',\rho'}} \Big(\sum_{\chi (\mathrm{mod} \thinspace q_{1})} \int_{t-\gamma' H}^{t+\gamma' H} \big| D[g](\frac{1}{2}+it',\chi)\big|dt'\Big)^{2}dt \end{equation} $$\ll_{\epsilon,B}q_{1}|\gamma'|^{2}H^{2}X(\log X)^{-\frac{B}{2.5}}.$$

\end{Lemma}

\begin{proof}
The proof of this basically follows from the arguments in \cite[Proposition 6.1]{MRT1}. Let $J'=[\rho' \gamma' X, \rho'^{-1} \gamma' X].$ Consider the Type $\rm{\RomanNumeralCaps{2}}$ case. Since $g = \varrho \ast \varsigma$, we have
\begin{equation} D[g](\frac{1}{2}+it',\chi)=D[\varrho](\frac{1}{2}+it',\chi) D[\varsigma](\frac{1}{2}+it',\chi). \end{equation}
By the Cauchy-Schwarz inequality, we have
\begin{equation}\begin{split}\Big(\sum_{\chi  (\mathrm{mod} \thinspace q_{1})} \int_{t-\gamma' H}^{t+\gamma' H} \big|D[g](\frac{1}{2}+it',\chi)\big|dt'\Big)^{2} &\ll \Big(\sum_{\chi(\mathrm{mod} \thinspace q_{1})} \int_{t-\gamma' H}^{t+\gamma' H} \big|D[\varrho](\frac{1}{2}+it',\chi)\big|^{2}dt' \Big) \\& \times \Big( \sum_{\chi (\mathrm{mod} \thinspace q_{1})} \int_{t-\gamma' H}^{t+\gamma' H} \big|D[\varsigma](\frac{1}{2}+it',\chi)\big|^{2}dt'\Big). \nonumber \end{split}\end{equation}
Using the mean value theorem (\cite{MRT1}, Lemma 2.10) and the Fubini theorem, we see that
\begin{equation}\begin{split}
&\sum_{\chi(\mathrm{mod} \thinspace q_{1})} \int_{t-\gamma' H}^{t+\gamma' H} \big|D[\varrho](\frac{1}{2}+it',\chi)\big|^{2}dt' \\&\ll (q_{1}\gamma' H + N) (\log q_{1}\gamma' HN)^{3}\sum_{N \leq n \leq 2N} \frac{\big(d_{2}(n)\log(2+n)\big)^{6}}{n},\\&
\int_{ I_{\gamma',\rho'}} \sum_{\chi (\mathrm{mod} \thinspace q_{1})} \int_{t-\gamma' H}^{t+\gamma' H} | D[\varsigma](\frac{1}{2}+it',\chi)|^{2}dt'dt \\& \ll (\log q_{1}\gamma' HN)^{3}|\gamma' H|(q_{1}\rho'^{-1}\gamma' X +M)\sum_{M \leq m \leq 2M} \frac{\big(d_{2}(m)\log(2+m)\big)^{6}}{m}.\end{split} \nonumber\end{equation}
Applying Shiu's Theorem (see \cite[Lemma 2.3,(i)]{MRT2}), we have
\begin{equation}\label{logpowerterm} (\log q_{1}\gamma' HN)^{3}\sum_{N \leq n \leq 2N} \frac{\big(d_{2}(n)\log(2+n)\big)^{6}}{n} \ll (\log X)^{72}. \end{equation}
Therefore, the left-hand side of \eqref{4cd} is bounded by 
$$\ll_{\epsilon}q_{1}(\rho'^{-1} q_{1} \gamma' + \frac{Q^{\frac{1}{2}}N}{H}+\frac{1}{N}+ \frac{1}{q_{1}\gamma' H})\gamma'^{2}H^{2}X(\log X)^{144} \ll_{\epsilon} q_{1}\rho'^{-1} Q^{-1} \gamma'^{2}H^{2}X(\log X)^{144}.$$ 
Consider the Type $\rm{d_{1}}$ case.
Let $g=\varrho \ast \varsigma \ast 1_{\{1\}}.$ So $$D[g](\frac{1}{2}+it', \chi)=D[\varrho](\frac{1}{2}+it', \chi)D[\varsigma](\frac{1}{2}+it', \chi)D[1_{\{1\}}](\frac{1}{2}+it', \chi).$$
By the Cauchy-Schwarz inequality, we have
\begin{equation} \begin{split} 
\Big(\sum_{\chi (\mathrm{mod} \thinspace q_{1})} \int_{t-\gamma' H}^{t+\gamma' H} |&D[g](\frac{1}{2}+it',\chi)|dt'\Big)^{2} \ll \Big( \sum_{\chi(\mathrm{mod} \thinspace q_{1})} \int_{t-\gamma' H}^{t+\gamma' H} |D[\varrho](\frac{1}{2}+it',\chi)|^{2}dt' \Big) \\&\times \Big(\sum_{\chi (\mathrm{mod} \thinspace q_{1})} \int_{t-\gamma' H}^{t+\gamma' H}  |D[\varsigma](\frac{1}{2}+it',\chi)D[1_{\{1\}}](\frac{1}{2}+it',\chi)|^{2} dt'\Big).  \nonumber
\end{split} \end{equation}
Using the bound 
 \begin{equation} \sum_{\chi(\mathrm{mod} \thinspace q_{1})} \int_{t-\gamma' H}^{t+\gamma' H} |D[\varrho](\frac{1}{2}+it',\chi)|^{2}dt' \ll (q_{1}\gamma' H + N) (\log X)^{72}, \nonumber\end{equation}
the left-hand side of \eqref{4cd} is bounded by 
$$(q_{1}\gamma' H+N)(\log X)^{72}\int_{ I_{\gamma',\rho'}} \sum_{\chi  (\mathrm{mod} \thinspace q_{1})} \int_{t-\gamma' H}^{t+\gamma' H} |D[\varsigma](\frac{1}{2}+it',\chi)D[1_{\{1\}}](\frac{1}{2}+it',\chi)|^{2} dt'dt.$$
By the pigeonhole principle, the above term is bounded by 
$$ (q_{1}\gamma' H+N)(\log X)^{73} \gamma' H\int_{\frac{T}{2}}^{T} \sum_{\chi  (\mathrm{mod} \thinspace q_{1})}|D[\varsigma](\frac{1}{2}+it,\chi)|^{2}|D[1_{\{1\}}](\frac{1}{2}+it,\chi)|^{2} dt $$
for some $T \in J'\:$ where the absolute constant depends on $\epsilon.$
Using the fourth moment estimate (see \cite[Corollary 2.12]{MRT1}) and the condition $q_{1} \leq Q =(\log X)^{B},$ we see that
\begin{equation} \begin{split}\sum_{\chi  (\rm{mod} \thinspace \it{q_{1}})} \int_{\frac{T}{2}}^{T} |D[\varsigma](\frac{1}{2}+it,\chi)|^{4}dt &\ll_{\epsilon} q_{1}T(1+\frac{q_{1}^{2}}{T^{2}}+\frac{M^{2}}{T^{4}})e^{(\log q_{1})^{\frac{1}{2}}}(\log q_{1}T)^{404} \\ &\ll_{\epsilon} q_{1}T(1+\frac{q_{1}^{2}}{T^{2}}+\frac{M^{2}}{T^{4}})(\log T)^{405}\end{split}\end{equation}
(In \cite[Corollary 2.12]{MRT1}, the factor $e^{(\log q_{1})^{\frac{1}{2}}}(\log q_{1}T)^{404}$ was replaced by $(\log X)^{O(1)},$ because this bound was sufficient. This bound would also be sufficient for our argument. For the precise results, We refer the reader to \cite[Theorem 6]{Rama1}, \cite[Lemma 9]{BHP}).
By the trivial bound $|D[1_{\{1\}}](\frac{1}{2}+it,\chi)| \leq 1,$ we see that
$$\sum_{\chi  (\mathrm{mod} \thinspace q_{1})} \int_{\frac{T}{2}}^{T} |D[1_{\{1\}}](\frac{1}{2}+it,\chi)|^{4}dt \ll_{\epsilon}  q_{1}T .$$
By the Cauchy-Schwarz inequality, we see that
$$ \sum_{\chi  (\mathrm{mod} \thinspace q_{1})} \int_{t-\gamma' H}^{t+\gamma' H}  |D[\varsigma](\frac{1}{2}+it',\chi)|^{2}|D[1_{\{1\}}](\frac{1}{2}+it',\chi)|^{2} dt' \ll_{\epsilon} q_{1}T(1+\frac{q_{1}}{T}+\frac{M}{T^{2}}) (\log X)^{202.5}.$$
From the conditions in Lemma \ref{Lemma 4.1}  and \eqref{Typed1}, we have $$ X^{\frac{1}{2}-2\epsilon}\leq T \leq \frac{1}{q_{1}Q^{\frac{1}{2}}}X,\;  \frac{1}{2}X^{-\frac{1}{2}-\epsilon}\leq \gamma' \leq \frac{1}{q_{1}Q},\; N\leq X^{\epsilon^{2}},\; \rho' X \leq \frac{T}{\gamma'} \leq \rho^{'-1}X.$$ Therefore, the left-hand side of \eqref{4cd} is bounded by 
\begin{equation}\begin{split} &\ll_{\epsilon} q_{1}(q_{1}T+\frac{TN}{\gamma' H})(1+\frac{q_{1}}{T}+\frac{M}{T^{2}})\gamma'^{2} H^{2}  (\log X)^{275.5} \\&\ll_{\epsilon} q_{1}\gamma'^{2}H^{2} (\log X)^{275.5} (q_{1}T+ q_{1}^{2}+q_{1}\frac{M}{T}+\frac{TN}{\gamma' H}+ \frac{q_{1}N}{\gamma' H}+\frac{X}{\gamma' HT})
\\&\ll_{\epsilon} q_{1}\gamma'^{2}H^{2} \frac{X}{(\log X)^{-275.5+\frac{B}{2}}}.\end{split} \nonumber \end{equation} 
\end{proof}
\begin{Rem} For Type $\rm{d_{2}}$ cases (see \cite[Proposition 6,1]{MRT1}), the dummy factor $1_{\{1\}}$ is replaced by other factors. In the proof of  \cite[Proposition 6,1]{MRT1}), Both the Type $\rm{d_{1}}$ case and the Type $\rm{d_{2}}$ case are treated by some fourth moment estimates, and the authors unified the proof of the Type $\rm{d_{1}}$ case into the Type $\rm{d_{2}}$ case by adding the dummy factor to the Type $\rm{d_{1}}$ case. In general, having additional factors gives us more flexibility. For example, see \cite[Lemma A.1]{MRT1}).
\end{Rem}

\subsection{Bounding the first term of \eqref{longintegral}}
\begin{Lemma}\label{Lemma 4.5} Let $(a,q)=1,$ and let $\gamma, \rho'$ be real numbers such that \\
$\gamma \in [\frac{1}{2}X^{-\frac{1}{2}-\epsilon},$ $\frac{1}{qQ}], \rho'=Q^{-\frac{1}{2}}.$ Let 
\begin{equation}
I_{\gamma,\rho'}:=  \{ t \in \mathbb{R}: \rho'|\gamma|X \leq |t| \leq \frac{|\gamma|X}{\rho'} \}.
\end{equation}
Then 
 \begin{equation}\begin{split} \frac{d_{2}(q)^{4}}{q|\gamma|^{2}H^{2}}& \sup_{q=q_{0}q_{1}} \int_{ I_{\gamma,\rho'}} \Big(\sum_{\chi(\mathrm{mod} \thinspace q_{1})}\int_{t-|\gamma|H}^{t+|\gamma|H} \big|\sum_{\frac{X}{q_{0}} \leq n \leq \frac{2X}{q_{0}}} \frac{b(q_{0}n)\chi(n)}{n^{\frac{1}{2}+it'}}\big| dt'\Big)^{2}dt \\&\ll_{\epsilon,B} X(\log X)^{-\frac{B}{3}+1}.\end{split}\end{equation} 
\end{Lemma}
\begin{proof}

For simplicity, we use $b$ to denote the arithmetic function $n \rightarrow b(n)$ and  $\lambda_{f,2}$ to denote the arithmetic function $n \rightarrow \lambda_{f}(n^{2}).$ Consider $b1_{(\frac{X}{2},2X]}\thinspace$ instead of $\thinspace b1_{(X,2X]}.$
$b1_{(\frac{X}{2},2X]}$ can be written as the following sum
\begin{equation}
\begin{split}
b1_{(\frac{X}{2},2X]}(q_{0}n)= &\sum_{0 \leq l \leq \log_{2}X}  \lambda_{f,2}1_{[2^{l},2^{l+1}]}\ast 1_{[\frac{X}{2^{l+1}},\frac{X}{2^{l}}]}(q_{0}n)
\\=&\sum_{0\leq l \leq \log_{2}(X^{-\epsilon^{2}}H)}  \lambda_{f,2}1_{[2^{l},2^{l+1}]} \ast 1_{[\frac{X}{2^{l+1}},\frac{X}{2^{l}}]}(q_{0}n).
\end{split}
\nonumber
\end{equation}
Let $l_{1}$ be the largest integer such that 
$2^{l_{1}} \leq X^{\epsilon^{2}}$, $l_{2}$ be the largest integer such that $2^{l_{2}} \leq X^{-\epsilon^{2}}H$.  Then \begin{description}
\item[\normalfont{For} $0\leq l \leq l_{1}$]$\lambda_{f,2}1_{[2^{l},2^{l+1}]}\ast 1_{[\frac{X}{2^{l+1}},\frac{X}{2^{l}}]}$ are Type $\rm{d_{1}}$ sums $\varrho \ast \varsigma$ such that $$\varrho= \lambda_{f,2}1_{[2^{l},2^{l+1}]},\; \varsigma= 1_{[\frac{X}{2^{l+1}},\frac{X}{2^{l}}]}.$$ 
\item[\normalfont{For} $ l_{1} < l \leq l_{2}$]  $\lambda_{f,2}1_{[2^{l},2^{l+1}]} \ast 1_{[\frac{X}{2^{l+1}},\frac{X}{2^{l}}]}$ are Type $\rm{\RomanNumeralCaps{2}}$ sums $\varrho \ast \varsigma$ such that $$\varrho= \lambda_{f,2}1_{[2^{l},2^{l+1}]},\; \varsigma= 1_{[\frac{X}{2^{l+1}},\frac{X}{2^{l}}]}.$$\end{description}
Notice that
\begin{equation}\label{4chilambda}\sum_{n=1}^{\infty} \frac{\chi(n)\lambda_{f,2}1_{[2^{l},2^{l+1}]}\ast  1_{[\frac{X}{2^{l+1}},\frac{X}{2^{l}}]}(q_{0}n)}{n^{\frac{1}{2}+it}}= \sum_{d|q_{0}}\sum_{(n,q_{0})=d} \frac{\chi(n)\lambda_{f,2}1_{[2^{l},2^{l+1}]}\ast  1_{[\frac{X}{2^{l+1}},\frac{X}{2^{l}}]}(q_{0}n)}{n^{\frac{1}{2}+it}}.
\end{equation}
For simplicity, assume $q_{0}$ is square-free. By using the multiplicativity of $\lambda_{f,2}$, the absolute value of \eqref{4chilambda} is bounded by 
\begin{equation}
\sum_{d|q_{0}} \Big|\lambda_{f,2}1_{[2^{l},2^{l+1}]}\ast  1_{[\frac{X}{2^{l+1}},\frac{X}{2^{l}}]}(dq_{0})\frac{\chi(d)}{d^{\frac{1}{2}+it}}\Big|\Big|\sum_{(k,\frac{q_{0}}{d})=1} \frac{\chi(k)\lambda_{f,2}1_{[2^{l},2^{l+1}]}\ast  1_{[\frac{X}{2^{l+1}},\frac{X}{2^{l}}]}(k)}{k^{\frac{1}{2}+it}}\Big|.   \end{equation}
Using the Deligne bound and the divisor bound, we have
\begin{equation}\Big|\lambda_{f,2}1_{[2^{l},2^{l+1}]}\ast  1_{[\frac{X}{2^{l+1}},\frac{X}{2^{l}}]}(dq_{0})\frac{\chi(d)}{d^{\frac{1}{2}+it}}\Big| \ll d_{2}(q_{0})^{3}. \end{equation}
Therefore, $\sum_{\frac{X}{q_{0}} \leq n \leq \frac{2X}{q_{0}}} \frac{b(q_{0}n)\chi(n)}{n^{\frac{1}{2}+it'}}$ is bounded by a linear combination (with coefficients of size $d_{2}(q_{0})^{3}$) of $O\big((\log_{2}X) d_{2}(q_{0})\big)$ absolute values of the following terms
\begin{equation}\label{4conditionk}\Big|\sum_{(k,\frac{q_{0}}{d})=1} \frac{\chi(k)\lambda_{f,2}1_{[2^{l},2^{l+1}]}\ast  1_{[\frac{X}{2^{l+1}},\frac{X}{2^{l}}]}(k)}{k^{\frac{1}{2}+it}}\Big|.\end{equation}
Let $\chi_{\frac{q_{0}}{d},0}$ be the principal character modulo $\frac{q_{0}}{d}.$ By replacing the character $\chi$ with $\chi_{\frac{q_{0}}{d},0}\:\chi,$ we can remove the condition on $k$ under the summation \eqref{4conditionk}. So we use the conductor $\frac{q}{d}$ instead of $q_{1}.$  
Applying Lemma \ref{Lemma 4.3}  on  $\chi_{\frac{q_{0}}{d},0}(k) \chi(k)\lambda_{f,2}1_{[2^{l},2^{l+1}]}\ast  1_{[\frac{X}{2^{l+1}},\frac{X}{2^{l}}]}(k),$ we see that 
\begin{equation}
d_{2}(q_{0})^{6} \int_{ I_{\gamma,\rho'}} \Big(\sum_{\chi (\mathrm{mod}\thinspace q_{1})} \int_{t-\gamma H}^{t+\gamma H} \Big|\sum_{k=1}^{\infty} \frac{\chi_{\frac{q_{0}}{d},0}(k)\chi(k)\lambda_{f,2}1_{[2^{l},2^{l+1}]}\ast  1_{[\frac{X}{2^{l+1}},\frac{X}{2^{l}}]}(k)}{k^{\frac{1}{2}+it}}\Big|dt'\Big)^{2}dt \nonumber\\ \end{equation} $$\ll_{\epsilon,B}d_{2}(q_{0})^{6}\frac{q}{d}|\gamma|^{2}H^{2}X(\log X)^{-\frac{B}{2.5}}.$$
In the same manner we can get the same upper bound when $q_{0}$ is not square-free.  
Hence, the first term in Lemma \ref{Lemma 4.3}  is bounded by $X(\log X)^{-\frac{B}{3}+1}.$ \end{proof}
\subsection{Bounding the second term of \eqref{longintegral}}
\begin{Lemma}\label{Lemma 4.6} Let $\gamma \in [\frac{1}{2}X^{-\frac{1}{2}-\epsilon}, \frac{1}{qQ}],$ and let $\rho'=Q^{-\frac{1}{2}}.$ Then
\begin{equation}\label{42}
\Big(\rho' +\frac{1}{|\gamma|H}\Big)^{2}\int_{X}^{2X} \big(\frac{1}{H} \sum_{x \leq n \leq x+H}|b(n)|\big)^{2}dx\ll_{B} \rho'^{2}X(\log X)^{33}.
\end{equation}
\end{Lemma}
\begin{proof}
By the Cauchy-Schwarz inequality, we have
\begin{equation}
\begin{split}
\int_{X}^{2X} \big(\frac{1}{H} \sum_{x \leq n \leq x+H}|b(n)|\big)^{2}dx&\ll\frac{1}{H} \int_{X}^{2X} \sum_{x \leq n \leq x+H}\big|\sum_{m|n}\lambda_{f}(m^{2}) 1_{[HX^{-\epsilon^{2}},2X]}(m)\big|^{2} dx
\\&\ll \sum_{X \leq n \leq 2X}\big|\sum_{m|n}\lambda_{f}(m^{2})1_{[HX^{-\epsilon^{2}},2X]}(m)\big|^{2}.
\end{split}
\nonumber
\end{equation}
By the Deligne bound, $\lambda_{f}(m^{2})\leq d_{2}(m^{2}).$ Using the inequality $d_{2}(mn) \ll d_{2}(m)d_{2}(n),$ 
\begin{equation}
\begin{split} 
\sum_{X \leq n \leq 2X}\big|\sum_{m|n}\lambda_{f}(m^{2})1_{[HX^{-\epsilon^{2}},2X]}\big|^{2} &\ll \sum_{X \leq n \leq 2X} d_{2}(n) (\sum_{m|n}d_{2}(m^{2})^{2}1_{[HX^{-\epsilon^{2}},2X]}(m))
\\& \ll \sum_{HX^{-\epsilon^{2}} \leq m \leq 2X } d_{2}(m^{2})^{2} d_{2}(m)\sum_{\frac{X}{m} \leq n \leq \frac{2X}{m}} d_{2}(n)
\\& \ll \sum_{HX^{-\epsilon^{2}} \leq m \leq 2X} d_{2}(m^{2})^{2} d_{2}(m) \frac{X}{m} \log X
\\& \ll X(\log X)^{2} \max_{HX^{-\epsilon^{2}}\leq M \leq 2X} \frac{\sum_{M \leq m \leq 2M} d_{2}(m)^{5}}{M}.
\end{split}\nonumber\end{equation}
Since \begin{equation}\frac{\sum_{M<m<2M} d_{2}(m)^{5}}{M} \ll (\log M)^{31}, \;\frac{1}{\gamma H} \leq 2X^{\frac{1}{2}+\epsilon}X^{-\frac{2}{3}-\epsilon}=2X^{-\frac{1}{6}}, \nonumber \end{equation}
\eqref{42} is bounded by $\rho'^{2}X(\log X)^{33}.$
\end{proof}
\begin{Proof-non3}
Recall that $\rho= (\log X)^{s}, \rho'=Q^{-\frac{1}{2}}=(\log X)^{-\frac{B}{2}}.$ Let $B>6k+3s+2020.$
Combining Lemma \ref{Lemma 4.2}, Lemma \ref{Lemma 4.5}  and Lemma \ref{Lemma 4.6} , we get 
\begin{equation} \begin{split}\int_{m_{2} \cap [\theta -\frac{1}{2H}, \theta+\frac{1}{2H}]}
 |S_{f}(\alpha)|^{2} d\alpha  &\ll_{f,B,s,k,\epsilon} X^{1-\epsilon}+ X(\log X)^{-\frac{B}{3}+1}+ (\rho')^{2}X (\log X)^{33}\\&\ll_{f,B,s,k,\epsilon} \rho X (\log X)^{-2k-1}.  \end{split} \nonumber \end{equation} \qed
\end{Proof-non3}
\begin{Rem}Applying the methods in Section 4, we can deduce Theorem \ref{Theorem 1.1} for $X^{\frac{5}{6}+\epsilon}<H<X^{1-\epsilon}.$     Precisely, one needs to replace the range of $\gamma \in [\frac{1}{2}X^{-\frac{1}{2}-\epsilon}, \frac{1}{q_{1}Q}]$ with $ \gamma \in  [X^{-\frac{5}{6}-\epsilon}, \frac{1}{q_{1}Q}],$ and one needs to replace the range of $H \in [X^{\frac{2}{3}+\epsilon},X^{1-\epsilon}]$ with $H \in [X^{\frac{5}{6}+\epsilon},X^{1-\epsilon}]$ in Lemma \ref{Lemma 4.3} .
\end{Rem}

\section{$m_{1}$-estimate}
\noindent In this section we will prove the  $m_{1}$-estimate in Proposition \ref{Propm}.
\begin{Prop}\label{Proposition 5.1} Let $\epsilon>0$ be sufficiently small. Let $1 \leq a \leq q \leq (\log X)^{B}$ for sufficiently large $B>0,$ and let $(a,q)=1.$  Then  
$$\int_{X^{-\frac{5}{6}-2\epsilon}}^{\frac{1}{2}X^{-\frac{1}{2}-\epsilon}} |S_{f}(\frac{a}{q}\pm \alpha)|^{2} d\alpha \ll_{f,B,\epsilon} X^{1-\frac{\epsilon}{2}}.$$
\end{Prop}
\begin{Proof-non2}
It follows from Proposition \ref{Proposition 5.1} that 
\begin{equation} \begin{split}\int_{m_{1}\cap [\theta-\frac{1}{2H}, \theta+\frac{1}{2H}]} |S_{f}(\alpha)|^{2}d\alpha &\ll \sum_{q=1}^{(\log X)^{B}} \sum_{1\leq a \leq q \atop(a,q)=1} \int_{X^{-\frac{5}{6}-2\epsilon}}^{\frac{1}{2}X^{-\frac{1}{2}-\epsilon}} |S_{f}(\frac{a}{q}\pm \beta)|^{2}d\beta \\&\ll_{f,B,\epsilon} (\log X)^{2B} X^{1-\frac{\epsilon}{2}}.  \end{split} \nonumber \end{equation}
By taking an arbitrarily small $\epsilon,$ the above term is bounded by
$O_{f,B,s,k,\epsilon} (\rho X (\log X)^{-2k-1}).$
\qed 
\end{Proof-non2} 
\noindent Let us briefly sketch the proof of Proposition \ref{Proposition 5.1}. To prove Proposition \ref{Proposition 5.1}, it is sufficient to show that   $$\int_{X^{-\frac{5}{6}-2\epsilon}} ^{\frac{1}{2}X^{-\frac{1}{2}-\epsilon}}\big|S_{f}(\frac{a}{q}\pm \alpha)\big|^{2}\big(\big|\frac{\sin\pi X^{\frac{1}{2}+\epsilon} \alpha}{\pi X^{\frac{1}{2}+\epsilon} \alpha}\big|^{2}-\big|\frac{\sin \pi X^{\frac{5}{6}+2\epsilon} \alpha}{\pi X^{\frac{5}{6}+2\epsilon} \alpha}\big|^{2} \big)d\alpha \ll_{f,B,\epsilon} X^{1-\frac{\epsilon}{2}} $$ instead of the original integral, since $$  \big|\frac{\sin\pi X^{\frac{1}{2}+\epsilon} \alpha}{\pi X^{\frac{1}{2}+\epsilon} \alpha}\big|^{2}-\big|\frac{\sin\pi X^{\frac{5}{6}+2\epsilon} \alpha}{\pi X^{\frac{5}{6}+2\epsilon} \alpha}\big|^{2} \geq \frac{1}{4}-\frac{1}{\pi^{2}}$$ for $\alpha \in [X^{-\frac{5}{6}-2\epsilon}, \frac{1}{2}X^{-\frac{1}{2}-\epsilon}]$ (see Remark \ref{Remark 5.2}).     
By the arguments in the proof of Gallagher's Lemma (see \cite[Lemma 1.9]{Montgomery1}), we have 
\begin{equation} \label{Gallagher} \begin{split}
\int_{-\infty}^{\infty} |S_{f}(\frac{a}{q}+\alpha)|^{2}\big|\frac{\sin\pi X^{\delta} \alpha}{\pi X^{\delta} \alpha}\big|^{2} d\alpha &= \int_{-\infty}^{\infty} \big|\frac{1}{X^{\delta}} \sum_{x \leq n \leq x +X^{\delta} \atop X \leq n\leq 2X}\lambda_{f}(n)^{2}e(\frac{a}{q}n)\big|^{2}dx\\&= \int_{X}^{2X}  \big|\frac{1}{X^{\delta}} \sum_{x \leq n \leq x +X^{\delta} \atop X \leq n \leq 2X} \lambda_{f}(n)^{2}e(\frac{a}{q}n)\big|^{2}dx \\&+ O_{f,B,\epsilon}(X^{\delta+\epsilon})
\end{split} \end{equation} for $\delta \in \{\frac{5}{6}+2\epsilon, \frac{1}{2}+\epsilon\}.$ In Lemma \ref{Lemma 5.5} we apply Parseval's identity to give the following asymptotic estimate
$$\int_{X}^{2X} \big|\frac{1}{X^{\delta}} \sum_{x \leq n \leq x +X^{\delta} \atop X \leq n\leq 2X}\lambda_{f}(n)^{2}e(\frac{a}{q}n)\big|^{2}dx  \sim \Big| \sum_{q=q_{0}q_{1}} \frac{\mu(q_{1})}{q_{0}\phi(q_{1})}\textrm{Res}_{s=1}\big(\frac{D[\lambda_{f}^{2}](s,q_{0};q_{1})}{s}\big)\Big|^{2}.$$  In Proposition \ref{Proposition 5.6} we show that
\begin{equation}\begin{split}&\int_{X}^{2X}  \big|\frac{1}{X^{\frac{1}{2}+\epsilon}} \sum_{x \leq n \leq x +X^{\frac{1}{2}+\epsilon} \atop X \leq n\leq 2X}\lambda_{f}(n)^{2}e(\frac{a}{q}n)\big|^{2} -\big|\frac{1}{X^{\frac{5}{6}+2\epsilon}} \sum_{x \leq n \leq x +X^{\frac{5}{6}+2\epsilon}\atop X \leq n\leq 2X}\lambda_{f}(n)^{2}e(\frac{a}{q}n)\big|^{2} dx\\&=O_{f,B,\epsilon}(X^{1-\frac{\epsilon}{2}}).\end{split}\end{equation} 
\subsection{Notation} From now on we take $c=\frac{5}{6}+2\epsilon,$ $d=\frac{1}{2}+\epsilon$ for some fixed small $\epsilon>0.$ We fix $a,q$ such that $(a,q)=1,$ $1 \leq a \leq q \leq Q=(\log X)^{B}$ for a sufficiently large constant $B>0.$ Let \begin{align} 
&K_{-\delta}(t):=|\frac{\sin\pi X^{\delta} t}{\pi X^{\delta} t}|^{2},\\\nonumber
&E_{\delta}(x):=\frac{1}{X^{\delta}} \sum_{x \leq n \leq x +X^{\delta}}\lambda_{f}(n)^{2}e(\frac{a}{q}n),\\\nonumber
&I_{1}:=[0,\frac{1}{\pi}X^{-c}],\\\nonumber
&J_{1}:=[\frac{1}{\pi}X^{-c}, X^{-c}]\cap \{ \alpha : K_{-d}(\alpha) < K_{-c}(\alpha)\}, \\\nonumber
&I_{2}:=[X^{-c},\frac{1}{2}X^{-d}],\\\nonumber
&I_{3}:=[\frac{1}{2}X^{-d},X^{\epsilon}],\\\nonumber
&J_{3}:=I_{3} \cap \{ \alpha : K_{-d}(\alpha) < K_{-c}(\alpha)\},\\\nonumber 
&I_{4}:=[X^{\epsilon}, \infty). 
\nonumber
\end{align}
\noindent When $P$ is a set, we denote by $\text{-}P$ the set 
$$\{ -x : x \in P \}.$$

\begin{Rem}\label{Remark 5.2} 
Since $\big|\frac{\sin t}{t}\big| \leq |\frac{1}{t}\big|$ for non-zero $t,$ we have $K_{-c}(\alpha) \leq \frac{1}{\pi^{2}}$ for  $\alpha \in [X^{-c}, \frac{1}{2}X^{-d}].$ By the Taylor  expansion of $\sin t$ about $0$, we have $$ \Big|\sin t - t + \frac{t^{3}}{6} \Big| \leq \frac{|t|^{5}}{120}$$ for $|t| \leq \frac{\pi}{2}.$ Therefore, for  $\alpha \in [X^{-c}, \frac{1}{2}X^{-d}],$ we see that  \begin{equation} \begin{split} K_{-d}(\alpha) &\geq \Big|\frac{\pi X^{d} \alpha - \frac{(\pi X^{d} \alpha)^{3}}{6} -\frac{(\pi X^{d} \alpha)^{5}}{120}}{\pi X^{d}\alpha}\Big|^{2}\\& \geq \Big(1- \frac{(\pi X^{d} \alpha)^{2}}{6} -\frac{(\pi X^{d} \alpha)^{4}}{120}\Big)^{2} \\& \geq \Big (1-\frac{\pi^{2}}{24} - \frac{\pi^{4}}{1920} \Big)^{2} \\& \geq \frac{1}{4}.\end{split} \nonumber \end{equation}
Combining the above inequalities,
$$K_{-d}(\alpha)-K_{-c}(\alpha) \geq \frac{1}{4}-\frac{1}{\pi^{2}}$$ for $\alpha \in [X^{-c}, \frac{1}{2}X^{-d}].$ Since $|K_{-d}(\alpha)|, |K_{-c}(\alpha)| \ll 1,$ we have
$$|K_{-d}(\alpha)-K_{-c}(\alpha)| \ll 1$$ for any $\alpha \in \mathbb{R}.$
\end{Rem}
\subsection{The asymptotic estimation of (5.1)}
In this subsection we use $\lambda_{f}^{2}$ to denote the function $n \rightarrow \lambda_{f}(n)^{2}.$ 
\begin{Lemma}\label{Lemma 5.3}
Let $T \geq 0.$ Then
 \begin{equation}\label{52} \int_{T}^{2T} \big|D[\lambda_{f}^{2}](\frac{1}{2}+it)\big|^{2}dt \ll_{f,\epsilon}  (1+T)^{\frac{11}{6}+\epsilon}.\end{equation}
\end{Lemma}
\begin{proof}
Let $$L_{f}^{(2)}(s):= \sum_{n=1}^{\infty} \big(\sum_{ml^{2}=n} \lambda_{f}(m^{2})\big)n^{-s}$$ for $\Re(s)>1$ and its analytic continuation elsewhere.
It is known that $$D[\lambda_{f}^{2}](s)=\frac{\zeta(s)}{\zeta(2s)}L_{f}^{(2)}(s).$$ (see \cite[(13.58), (13.60)]{Iwaniecbook2}).
Therefore, we see that 
\begin{equation}\label{5integral1} \int_{T}^{2T} \big|D[\lambda_{f}^{2}](\frac{1}{2}+it)\big|^{2}dt\ll\max_{t \in [T,2T]}\big|\frac{\zeta(\frac{1}{2}+it)}{\zeta(1+2it)}\big|^{2} \int_{T}^{2T} \big| L_{f}^{(2)}(\frac{1}{2}+it)\big|^{2}dt.\end{equation}
It is known that 
$$|\zeta(\frac{1}{2}+it)|\ll_{\epsilon} (1+|t|)^{\frac{1}{6}+\epsilon},$$
$$|\frac{1}{\zeta(1+2it)}|\ll ( \log (1+|t|))^{7}$$
(see \cite[Theorem 5.12]{Titchmarsh1}, \cite[(3.6.3)]{Titchmarsh1} respectively).
Let $$\Lambda_{f}^{(2)}(s):=\pi^{-\frac{3s}{2}} \Gamma(\frac{s+1}{2}) \Gamma(\frac{s+\kappa-1}{2}) \Gamma(\frac{s+\kappa }{2})L_{f}^{(2)}(s).$$
Then $\Lambda_{f}^{(2)}(s)$ satisfies the functional equaton
$$\Lambda_{f}^{(2)}(s)=\Lambda_{f}^{(2)}(1-s),$$ and $L_{f}^{(2)}(s)$ is entire  (for the details, see \cite[13.8]{Iwaniecbook2}).
By the approximate functional equation in \cite[Theorem 5.3]{IK1}, 
\begin{equation}\label{originalfunctional}\begin{split}&L_{f}^{(2)}(\frac{1}{2}+it)\\&= \sum_{n}\frac{ \big( \sum_{ml^{2}=n} \lambda_{f}(m^{2}) \big)}{n^{\frac{1}{2}+it}} V_{\frac{1}{2}+it}( \frac{n}{\sqrt{q_{f}(t)}}) +\epsilon_{f} \sum_{n}\frac{\overline{\big(\sum_{ml^{2}=n} \lambda_{f}(m^{2})\big)}}{n^{\frac{1}{2}-it}} V_{\frac{1}{2}-it}(\frac{n}{\sqrt{q_{f}(t)}}) \end{split}\end{equation} 
where $q_{f}(t):=(|\frac{3}{2}+it|+3)(|\kappa+\frac{1}{2}+it|+3)(|\kappa-\frac{1}{2}+it|+3),$ a constant $\epsilon_{f}$ such that $|\epsilon_{f}|=1,$ and a function $V_{s}(y)$ such that
$V_{s}(y) \ll_{\upsilon} (1+\frac{y}{\sqrt{q_{f}(t)}})^{-\upsilon}$ for any $\upsilon>0$ (in this proof, we take  $\upsilon=\frac{10000}{\epsilon}$). 
Note that if $ t^{\frac{1}{2}}\geq \kappa,$ then $\sqrt{q_{f}(t)} \ll t^{\frac{3}{2}+\frac{\epsilon}{2}},$ and if $ t^{\frac{1}{2}}<\kappa,$ then $\sqrt{q_{f}(t)} \ll \kappa^3 .$ 
If  $t^{\frac{1}{2}}\geq \kappa,$ we have 
\begin{equation}\label{aproxfunc}  L_{f}^{(2)}(\frac{1}{2}+it)= \sum_{n<t^{\frac{3}{2}+\frac{2\epsilon}{3}}} \frac{a_{f}^{(2)}(n)}{n^{\frac{1}{2}+it}}V_{\frac{1}{2}+it}(\frac{n}{\sqrt{q_{f}(t)}}) +\sum_{n<t^{\frac{3}{2}+\frac{2\epsilon}{3}}} \frac{b_{f}^{(2)}(n)}{ n^{\frac{1}{2}-it}}V_{\frac{1}{2}-it}(\frac{n}{\sqrt{q_{f}(t)}}) + O_{\epsilon}(t^{-1000})
\end{equation}
where $$a_{f}^{(2)}(n):=   \sum_{ml^{2}=n} \lambda_{f}(m^{2}) ,$$ $$b_{f}^{(2)}(n):= \epsilon_{f} \overline{\sum_{ml^{2}=n} \lambda_{f}(m^{2})}.$$
By the Deligne bound, we have
$$a_{f}^{(2)}(n), b_{f}^{(2)}(n) \ll d_{2}(n)^{3}.$$
Therefore, by the Dirichlet mean value theorem and the majorant principle (see \cite[Chapter 7, Theorem 3]{Montgomery10}), we get
\begin{equation} \begin{split} \int_{T}^{2T}  & \big| L_{f}^{(2)}(\frac{1}{2}+it)\big|^{2}dt \\&\ll_{f,\epsilon} \int_{T}^{2T} \Big| \sum_{n \leq t^{\frac{3}{2}+\frac{\epsilon}{2}}} \frac{|a_{f}^{(2)}(n)|}{ n^{\frac{1}{2}+it}} \Big|^{2} dt +  \int_{T}^{2T} \Big| \sum_{n \leq t^{\frac{3}{2}+\frac{\epsilon}{2}}} \frac{| b_{f}^{(2)}(n) |}{n^{\frac{1}{2}-it}} \Big|^{2} dt \\&\ll_{f,\epsilon} \big(T+ O(T^{\frac{3}{2}+\frac{\epsilon}{2}})\big) \sum_{n \leq T^{\frac{3}{2}+\frac{\epsilon}{2}}} \frac{ |a_{f}^{(2)}(n)|^{2} +|b_{f}^{(2)}(n)|^{2}}{n}  \\&\ll_{f,\epsilon} T^{\frac{3}{2}+\epsilon}\end{split} \end{equation}
for $\frac{T^{\frac{1}{2}}}{\sqrt{2}}\geq \kappa.$
If $t^{\frac{1}{2}} <\kappa ,$
\eqref{originalfunctional} is 
$$\ll_{f.\epsilon}\sum_{n \leq \kappa^{3+\epsilon}} \frac{ \big| \sum_{ml^{2}=n} \lambda_{f}(m^{2}) \big|}{n^{\frac{1}{2}}}+ \kappa^{-1000}.$$
Therefore, if $ \frac{T^{\frac{1}{2}}}{\sqrt{2}}<\kappa ,$ we have
$$ \int_{T}^{2T}  \big| L_{f}^{(2)}(\frac{1}{2}+it)\big|^{2}dt \ll_{f} 1.$$
Taking the above upper bounds to the right-hand side of \eqref{5integral1}, we have 
\begin{equation}\label{ttintegral}
 \int_{T}^{2T} \big|D[\lambda_{f}^{2}](\frac{1}{2}+it)\big|^{2}dt \ll_{f,\epsilon} (1+T)^{\frac{11}{6}+ \epsilon}.\end{equation}
\end{proof}
\begin{Rem}\label{Remark 5.4} 
In the same manner as Lemma \ref{Lemma 5.3}, we have
\begin{equation}\left|\label{5integral2}\int_{T}^{2T} |\mathcal{D}(\frac{1}{2}+it)|^{2}dt\right| \ll_{f,q_{0}, q_{1}, \chi,\epsilon} (1+|T|)^{\frac{11}{6}+\epsilon}\end{equation}
for all $\mathcal{D}(s) \in \{D[\lambda_{f}^{2}](s;q_{1}), D[\lambda_{f}^{2}](s,\chi), D[\lambda_{f}^{2}](s,\chi;q_{1}),D[\lambda_{f}^{2}](s,\chi,q_{0}) : q_{0}q_{1}=q,\\ \chi (\textrm{mod} \; q_{1}), \tau(\chi) \neq \rm{0} \}$ 
(for the notations, see Subsection 2.1). These Dirichlet series come from the transform in Lemma \ref{Lemma 3.1}. Therefore, we don't need to consider some characters $\chi$ such that $\tau(\chi)=0.$ In \cite{PYOUNG} the Weyl bounds for Dirichlet L-functions are proved. For the lower bounds of Dirichlet L-functions at the edge of the critical strip, see \cite[Theorem 11.4]{MontgomeryVaughan} (If $L(s,\chi)$ has an exceptional zero $\beta$ such that $|1+2it-\beta|\leq \frac{1}{\log q},$ then one can apply the bound of Siegel $1-\beta \gg_{\epsilon} q^{-\epsilon}$ for any $\epsilon>0.$ See \cite[Corollary 11.15]{MontgomeryVaughan}). Since the absolute constant in \eqref{5integral2} has polynomial growth, these are bounded by some power of $\log X.$ When $T<0,$ we use the fact that $\left|\mathcal{D}(\frac{1}{2}+it)\right|=\left|\mathcal{D}(\frac{1}{2}-it)\right|.$ Therefore, we have
\begin{equation}\label{55integral2}\left|\int_{T}^{2T} |\mathcal{D}(\frac{1}{2}+it)|^{2}dt\right| \ll_{f,\epsilon} (\log X)^{\kappa_{1}} (1+|T|)^{\frac{11}{6}+\epsilon}\end{equation}
 for some constant $\kappa_{1}$ depending on $B.$ 
\end{Rem}
\begin{Lemma}\label{NLemma 5.5} Let $T \geq 0,$ and $t \in [T,2T].$ Then 
\begin{equation}\label{5integral22} |\mathcal{D}(\frac{1}{2}+it)| \ll_{f,\epsilon} (\log X)^{\kappa_{1}} (1+T)^{\frac{11}{12}+\epsilon} \end{equation}
for all $ \mathcal{D}(s) \in \{D[\lambda_{f}^{2}](s;q_{1}), D[\lambda_{f}^{2}](s,\chi), D[\lambda_{f}^{2}](s,\chi;q_{1}),D[\lambda_{f}^{2}](s,\chi,q_{0}) : q_{0}q_{1}=q,\\ \chi (\mathrm{mod} \; q_{1}), \tau(\chi) \neq 0 \}.$
\end{Lemma}
\begin{proof} As mentioned in the proof of Lemma \ref{Lemma 5.3} and Remark \ref{Remark 5.4}, we have the Weyl bounds for the Riemann zeta function and Dirichlet L-functions.
The proof comes from the direct application of the convexity bounds for $L_{f}^{(2)}(\frac{1}{2}+it)$ and its twisted one. By an application of \cite[Lemma 11.9]{Titchmarsh1}, one can get the same bound.     
\end{proof}
\begin{Lemma}\label{Lemma 5.5}
Let $\epsilon>0$ be sufficiently small. Then
for any $\delta\geq \frac{1}{2}+\epsilon,$
\begin{equation}
\int_{X}^{2X} |E_{\delta}(x)|^{2} dx =\int_{X}^{2X} \big|\frac{W(q,x+X^{\delta})-W(q,x)}{X^{\delta}}\big|^{2}dx +O_{f,B,\epsilon}(X^{\frac{3}{2}-\delta+\frac{\epsilon}{500}})
\end{equation}
where $W(q,y):=\textrm{Res}_{s=1} \big(\sum_{q=q_{0}q_{1}} \frac{\mu(q_{1})}{\phi(q_{1})}\frac{D[\lambda_{f}^{2}](s,q_{0};q_{1})}{s} (\frac{y}{q_{0}})^{s}\big).$
\end{Lemma}
\begin{proof} 
By Lemma \ref{Lemma 3.1}, we see that 
$$\sum_{n=1}^{\infty}\lambda_{f}(n)^{2}e(\frac{a}{q}n)n^{-s}= \sum_{q=q_{0}q_{1}} \frac{1}{\phi(q_{1})q_{0}^{s}
} \sum_{\chi (\mathrm{mod} \thinspace q_{1})} \tau(\bar{\chi})\chi(a) \sum_{n=1}^{\infty} \lambda_{f}(q_{0}n)^{2}\chi(n)n^{-s}$$ for $\Re(s)>1.$
This yields
\begin{equation} \begin{split}\sum_{n=1}^{\infty} \lambda_{f}(n)^{2}e(\frac{a}{q}n)n^{-s}= \sum_{q=q_{0}q_{1}}& \frac{\mu(q_{1})}{\phi(q_{1})q_{0}^{s}}D[\lambda_{f}^{2}](s,q_{0};q_{1}) \\&+\sum_{q=q_{0}q_{1}}  \frac{1}{\phi(q_{1})q_{0}^{s}} \sum_{\chi (\mathrm{mod} \thinspace q_{1}) \atop \chi \neq \chi_{0}} \tau(\bar{\chi})\chi(a)D[\lambda_{f}^{2}](s,\chi,q_{0})\end{split} \nonumber \end{equation}
where $\chi_{0}$ is the principal character modulo $q_{1}.$
Let $$\bigtriangleup(x):= \sum_{1\leq n \leq x} \lambda_{f}(n)^{2}e(\frac{a}{q}n)-W(q,x).$$
%where $W(q,x):=\textrm{Res}_{s=1} \big(\sum_{q=q_{0}q_{1}} \frac{\mu(q_{1})}{\phi(q_{1})}\frac{D[\lambda_{f}^{2}](s,q_{0};q_{1})}{s} (\frac{x}{q_{0}})^{s}\big).$
By Perron's formula, we see that
\begin{equation}\begin{split} \sum_{1 \leq n \leq x} \lambda_{f}(n)^{2}e(\frac{a}{q}n)=&\frac{1}{2\pi i} \Big( \sum_{q=q_{0}q_{1}} \frac{\mu(q_{1})}{\phi(q_{1})} \int_{1+\epsilon-iT}^{1+\epsilon+iT}\frac{D[\lambda_{f}^{2}](s,q_{0};q_{1})}{s}(\frac{x}{q_{0}})^{s}ds\\&+ \sum_{q=q_{0}q_{1}}  \frac{1}{\phi(q_{1})} \sum_{\chi (\mathrm{mod} \thinspace q_{1}) \atop \chi \neq \chi_{0}} \tau(\bar{\chi})\chi(a) \int_{1+\epsilon-iT}^{1+\epsilon+iT} \frac{D[\lambda_{f}^{2}](s,\chi,q_{0})}{s}(\frac{x}{q_{0}})^{s}ds\Big)\\&+O_{\epsilon}\big(\frac{x^{1+\epsilon}}{T}\sum_{n=1}^{\infty} \frac{|\lambda_{f}(n)^{2}|}{n^{1+\epsilon}}+\sum_{\frac{x}{2}\leq n \leq 2x \atop n \neq x} |\lambda_{f}(n)^{2}| \min(1,\frac{x}{T|x-n|})\big) \nonumber \end{split} \end{equation}
for any $T>0.$ By \eqref{Fomenko} and the Cauchy-Schwarz inequality, we have
$$ \sum_{\frac{x}{2}\leq n \leq 2x \atop n \neq x} |\lambda_{f}(n)^{2}| \min(1,\frac{x}{T|x-n|})\big) \ll_{f,\epsilon} \frac{x^{1+\epsilon}}{T^{\frac{1}{2}}}.$$
When $s_{0}=1+\epsilon+iT,$ 
\begin{equation}\begin{split} D[\lambda_{f}^{2}](s_{0},q_{0};q_{1})&=  \sum_{n=1 \atop (n,q_{1})=1}^{\infty}  \frac{\lambda_{f}(nq_{0})^{2}}{n^{1+\epsilon+iT}}
\\&\leq q_{0}^{1+\epsilon}\sum_{n=1}^{\infty}\frac{\lambda_{f}(n)^{2}}{n^{1+\epsilon}}
\\&\ll_{f,\epsilon}q_{0}^{1+\epsilon}.\end{split} \nonumber \end{equation} 
By Lemma \ref{NLemma 5.5}, the above inequality and the convexity principle,
\begin{equation}\label{5integral3}\int_{\frac{1}{2} \leq \sigma \leq 1} D[\lambda_{f}^{2}](\sigma+iT,q_{0};q_{1})  d\sigma \ll_{f,\epsilon}(\log X)^{\kappa_{2}}(1+T)^{\frac{11}{12}+\epsilon}\end{equation}
for some constant $\kappa_{2}$ depending on $B.$
Shifting the contour to the rectangular path $1+\epsilon +iT \rightarrow \frac{1}{2}+iT,  \frac{1}{2} +iT \rightarrow \frac{1}{2} -iT,  \frac{1}{2} -iT \rightarrow 1+\epsilon -iT$ and using the residue theorem, we have
\begin{equation}\label{5dlambda1} \begin{split}\int_{1+\epsilon-iT}^{1+\epsilon+iT}\frac{D[\lambda_{f}^{2}](s,q_{0};q_{1})}{s} (\frac{x}{q_{0}})^{s}ds &= \int_{\frac{1}{2}-iT}^{\frac{1}{2}+iT}\frac{D[\lambda_{f}^{2}](s,q_{0};q_{1})}{s} (\frac{x}{q_{0}})^{s}ds \\&+\textrm{Res}_{s=1}\big(\frac{D[\lambda_{f}^{2}](s,q_{0};q_{1})}{s} (\frac{x}{q_{0}})^{s}\big)+O_{f,B,\epsilon}(\frac{x^{1+\epsilon}}{T^{\epsilon}}).\end{split} \end{equation}
 By the similar argument of \eqref{5dlambda1},
$$\int_{1+\epsilon-iT}^{1+\epsilon+iT}\frac{D[\lambda_{f}^{2}](s,\chi,q_{0})}{s} (\frac{x}{q_{0}})^{s}ds = \int_{\frac{1}{2}-iT}^{\frac{1}{2}+iT}\frac{D[\lambda_{f}^{2}](s,\chi,q_{0})}{s} (\frac{x}{q_{0}})^{s}ds+O_{f,B,\epsilon}(\frac{x^{1+\epsilon}}{T^{\epsilon}}).$$
Therefore, we have \begin{equation}\label{dyad}\begin{split}\bigtriangleup(x)
= &\frac{1}{2\pi i} \sum_{q=q_{0}q_{1}}\Big(  \frac{\mu(q_{1})}{\phi(q_{1})} \int_{ \frac{1}{2}-iT}^{ \frac{1}{2}+iT} \frac{D[\lambda_{f}^{2}](s,q_{0};q_{1})}{s}(\frac{x}{q_{0}})^{s}ds \\&+  \frac{1}{\phi(q_{1})} \sum_{\chi (\mathrm{mod} \thinspace q_{1}) \atop \chi \neq \chi_{0}} \tau(\bar{\chi})\chi(a) \int_{\frac{1}{2}-iT}^{\frac{1}{2}+iT}  \quad \frac{D[\lambda_{f}^{2}](s,\chi,q_{0})}{s}  (\frac{x}{q_{0}})^{s}ds\Big) +O_{f,B,\epsilon}(\frac{x^{1+\epsilon}}{T^{\epsilon}}). \nonumber\end{split}\end{equation}
By \eqref{5integral2}, we see that if $T \geq 1,$ then 
\begin{equation}\label{newequation}\int_{T}^{2T}  \frac{|D[\lambda_{f}^{2}](\frac{1}{2}+it,q_{0};q_{1})|^{2}}{|(\frac{1}{2}+it)|^{2}}+\max_{\chi (\mathrm{mod} \thinspace q_{1}) \atop \chi \neq \chi_{0}, \tau(\chi) \neq 0} \frac{ |D[\lambda_{f}^{2}](\frac{1}{2}+it,\chi,q_{0})|^{2}}{|(\frac{1}{2}+it)|^{2}} dt \end{equation} \begin{equation} \begin{split} &\ll T^{-2}\int_{T}^{2T} |D[\lambda_{f}^{2}](\frac{1}{2}+it,q_{0};q_{1})|^{2}+\max_{\chi (\mathrm{mod} \thinspace q_{1}) \atop \chi \neq \chi_{0}, \tau(\chi) \neq 0} |D[\lambda_{f}^{2}](\frac{1}{2}+it,\chi,q_{0})|^{2} dt\\&\ll_{f,\epsilon} (\log X)^{\kappa_{1}}T^{-\frac{1}{6} +\epsilon}. \nonumber\end{split}\end{equation}
 By the similar argument in the proof of Lemma \ref{Lemma 5.3}, we have 
\begin{equation}\begin{split} &\int_{0}^{1}  \frac{|D[\lambda_{f}^{2}](\frac{1}{2}+it,q_{0};q_{1})|^{2}}{|(\frac{1}{2}+it)|^{2}}+\max_{\chi (\mathrm{mod} \thinspace q_{1}) \atop \chi \neq \chi_{0}, \tau(\chi) \neq 0} \frac{ |D[\lambda_{f}^{2}](\frac{1}{2}+it,\chi,q_{0})|^{2}}{|(\frac{1}{2}+it)|^{2}} dt
\\&\ll  \int_{0}^{1} |D[\lambda_{f}^{2}](\frac{1}{2}+it,q_{0};q_{1})|^{2}+\max_{\chi (\mathrm{mod} \thinspace q_{1}) \atop \chi \neq \chi_{0}, \tau(\chi) \neq 0}  |D[\lambda_{f}^{2}](\frac{1}{2}+it,\chi,q_{0})|^{2} dt
\\&\ll_{f} (\log X)^{\kappa_{3}} \nonumber \end{split}\end{equation}
for some constant $\kappa_{3} > \kappa_{1}.$ 
By summing the integrals over the intervals $[0,1),(-1,0],$ $[1,2),(-2,-1], [2,4)...,$ we have 
\begin{equation}\label{equation522}
\max_{q=q_{0}q_{1}}\Big( \int_{-\infty}^{\infty}\frac{|D[\lambda_{f}^{2}](\frac{1}{2}+it,q_{0};q_{1})|^{2}}{|(\frac{1}{2}+it)|^{2}} + \max_{\chi (\mathrm{mod} \thinspace q_{1}) \atop \chi \neq \chi_{0}, \tau(\chi) \neq \rm{0}}  \frac { |D[\lambda_{f}^{2}](\frac{1}{2}+it,\chi,q_{0})|^{2}}{|(\frac{1}{2}+it)|^{2}} dt\Big) =O_{f}( (\log X)^{ \kappa_{3}}).\end{equation}
Let $$F_{q_{0};q_{1}}(s):=\frac{D[\lambda_{f}^{2}](s,q_{0};q_{1})}{s},\; \widehat {F}_{q_{0};q_{1}}(x,T):= \frac{1}{2 \pi i} \int_{\frac{1}{2}- iT}^{\frac{1}{2}+iT} F_{q_{0};q_{1}}(s) x^{-s}ds,$$
$$F_{\chi,q_{0}}(s):=\frac{D[\lambda_{f}^{2}](s,\chi,q_{0})}{s},\; \widehat{F}_{\chi,q_{0}}(x,T):= \frac{1}{2 \pi i} \int_{\frac{1}{2}- iT}^{\frac{1}{2}+iT} F_{\chi,q_{0}}(s) x^{-s}ds.$$
We have
\begin{equation} \begin{split} \bigtriangleup(x)&= \lim_{T \rightarrow \infty} \frac{1}{2\pi i} \sum_{q=q_{0}q_{1}}\Big(  \frac{\mu(q_{1})}{\phi(q_{1})} \widehat{F}_{q_{0};q_{1}}(\frac{q_{0}}{x},T)+  \frac{1}{\phi(q_{1})} \sum_{\chi (\mathrm{mod} \thinspace q_{1}) \atop \chi \neq \chi_{0}} \tau(\bar{\chi})\chi(a) \widehat{F}_{\chi,q_{0}}(\frac{q_{0}}{x},T)\Big). \end{split} \nonumber \end{equation}
%Let's consider 
%\begin{equation} \int_{X}^{2X} | \widehat{F}_{q_{0};q_{1}}(\frac{q_{0}}{x},T_{0})|^{2} \frac{q_{0}}{x^{2}}dx. \end{equation}
%This is equals to
%\begin{equation}\begin{split} \int_{\frac{q_{0}}{2X}}^{\frac{q_{0}}{X}} | \widehat{F}_{q_{0};q_{1}}(y,T_{0})|^{2}dy
%&=\int_{\frac{q_{0}}{2X}}^{\frac{q_{0}}{X}} \widehat{F}_{q_{0};q_{1}}(y,T_{0}) \frac{1}{2\pi i} \int_{\frac{1}{2}-iT_{0}}^{\frac{1}{2}+iT_{0}} \overline{F_{q_{0};q_{1}}(s)}\frac{q_{0}}{y}y^{-\bar{s}} ds dy \\&=  \frac{1}{2\pi i} \int_{\frac{1}{2}-iT_{0}}^{\frac{1}{2}+iT_{0}}  \overline{F_{q_{0};q_{1}}(s)} \int_{\frac{q_{0}}{2X}}^{\frac{q_{0}}{X}} \widehat{F}_{q_{0};q_{1}}(y,T_{0})y^{-\bar{s}} ds dy.   \end{split}\end{equation}
%As a convexity bound, we have $|F_{q_{0};q_{1}}(s)| \ll_{\epsilon} X^{\epsilon}|s|^{\frac{3}{4}-1}.$
%Therefore, by the residue theorem,
%\begin{equation}\begin{split} \int_{\frac{q_{0}}{2X}}^{\frac{q_{0}}{X}} \widehat{F}_{q_{0};q_{1}}(y,T_{0})y^{-\bar{s}} dy =&  \int_{\frac{1}{2}-iT_{0}}^{\frac{1}{2}+iT_{0}}  F_{q_{0};q_{1}}(u) \int_{\frac{q_{0}}{2X}}^{\frac{q_{0}}{X}} y^{-u-\bar{s}} dy du \\&\ll_{\epsilon} |F_{q_{0};q_{1}}(s)|+ X^{\epsilon} T_{0}^{-\frac{1}{4}}
%\end{split}\end{equation}
By Parseval's identity(\cite{Ivicbook}, (A.5)), we see that
\begin{equation}\label{5Parseval}\begin{split}&\frac{1}{2\pi } \int_{-\infty}^{\infty} \Big|\frac{D[\lambda_{f}^{2}](\frac{1}{2}+it,q_{0};q_{1})}{\frac{1}{2}+it}\Big|^{2}dt= \lim_{T \rightarrow \infty} \int_{0}^{\infty} |\widehat{F}_{q_{0};q_{1}}(x,T)|^{2}dx,\\&
\frac{1}{2\pi } \int_{-\infty}^{\infty} \Big|\frac{D[\lambda_{f}^{2}](\frac{1}{2}+it,\chi,q_{0})}{\frac{1}{2}+it}\Big|^{2}dt= \lim_{T\rightarrow \infty} \int_{0}^{\infty} |\widehat{F}_{\chi,q_{0}}(x,T)|^{2}dx.\end{split}\end{equation}
By the Cauchy-Schwarz inequality, \eqref{equation522} and \eqref{5Parseval}, we get
\begin{equation}
\begin{split}\int_{0}^{\infty}& |\frac{\bigtriangleup(x)}{x}|^{2}dx \\&\ll_{f} \Big(\sum_{q=q_{0}q_{1}} \sum_{\chi (\mathrm{mod} \thinspace q_{1})} |\tau(\bar{\chi})|^{2} \Big) \Big(\sum_{q=q_{0}q_{1}} \sum_{\chi (\mathrm{mod} \thinspace q_{1})} (\log X)^{\kappa_{3}}\Big)
\\&\ll_{f} (\log X)^{\kappa_{3}+5B} \end{split}\end{equation} 
(recall that $|\tau(\bar{\chi})|\leq q^{\frac{1}{2}}$).
Thus, we have \begin{equation} \int_{X}^{2X} |\bigtriangleup(x)|^{2} dx \ll X^{2} \int_{0}^{\infty} |\frac{\bigtriangleup(x)}{x}|^{2}dx \ll_{f,B,\epsilon} X^{2+\frac{\epsilon}{1000}}. \end{equation}
Since $E_{\delta}(x)= \frac{\bigtriangleup(x+X^{\delta})-\bigtriangleup(x)}{X^{\delta}}+\frac{W(q,x+X^{\delta})-W(q,x)}{X^{\delta}},$ 
\begin{equation}\label{5integralerror1} \begin{split}
\int_{X}^{2X} &|E_{\delta}(x)|^{2} dx -\int_{X}^{2X} |\frac{W(q,x+X^{\delta})-W(q,x)}{X^{\delta}}|^{2}dx\\&
= O_{\epsilon}\Big(\int_{X}^{2X} \big|\frac{\max \big(\bigtriangleup(x),\bigtriangleup(x+X^{\delta})\big)}{X^{\delta}}\frac{W(q,x+X^{\delta})-W(q,x)}{X^{\delta}}\big|dx \\&
+ \int_{X}^{2X} \frac{\max \big(\bigtriangleup(x),\bigtriangleup(x+X^{\delta})\big)^{2}}{X^{2\delta}}dx\Big).
\end{split}
\end{equation}
By the calculation of $D_{q}$ (see \eqref{3dq}), we have
$$\big|\frac{W(q,x+X^{\delta})-W(q,x)}{X^{\delta}}\big|\ll_{f,B,\epsilon} X^{\frac{\epsilon}{1000}}.$$
Therefore, \eqref{5integralerror1} is bounded by
\begin{equation} \begin{split} & O_{f,B,\epsilon}\Big(X^{\frac{1}{2}+\frac{\epsilon}{1000}}(\int_{X}^{2X}  \frac{\max(\bigtriangleup(x),\bigtriangleup(x+X^{\delta}))^{2}}{X^{2\delta}} dx\big)^{\frac{1}{2}}+ \int_{X}^{2X} \frac{\max(\bigtriangleup(x),\bigtriangleup(x+X^{\delta}))^{2}}{X^{2\delta}}dx\Big)\\
&=O_{f,B,\epsilon}(X^{1+\frac{1}{2}-\delta+\frac{\epsilon}{500}}).
\end{split}
\end{equation} 
\end{proof}
\begin{Prop}\label{Proposition 5.6} Let $\epsilon>0$ be sufficiently small. Then
 $$\int_{X}^{2X} |E_{d}(x)|^{2}-|E_{c}(x)|^{2}dx =O_{f,B,\epsilon}(X^{1-\frac{\epsilon}{2}}).$$
\end{Prop}
\begin{proof}
 By Lemma \ref{Lemma 5.5}, it is sufficient to show that 
\begin{equation} \int_{X}^{2X} \big|\frac{W(q,x+X^{d})-W(q,x)}{X^{d}}\big|^{2}-\big|\frac{W(q,x+X^{c})-W(q,x)}{X^{c}}\big|^{2}dx =0. \nonumber\end{equation}
By the definition of $W(q,x),$ we have
 $$\frac{W(q,x+X^{d})-W(q,x)}{X^{d}}=\sum_{q=q_{0}q_{1}} \frac{\mu(q_{1})}{q_{0}\phi(q_{1})}\textrm{Res}_{s=1}\big(\frac{D[\lambda_{f}^{2}](s,q_{0};q_{1})}{s}\big).$$
$$\frac{W(q,x+X^{c})-W(q,x)}{X^{c}}=\sum_{q=q_{0}q_{1}} \frac{\mu(q_{1})}{q_{0}\phi(q_{1})}\textrm{Res}_{s=1}\big(\frac{D[\lambda_{f}^{2}](s,q_{0};q_{1})}{s}\big).$$
\end{proof}
\subsection{Proof of Proposition \ref{Proposition 5.1}}
In this subsection we denote $m(I)$ as the Lebesgue measure of $I.$

\begin{proof}
 By the Taylor expansion of $\sin t$ about $0$, we have   \begin{equation}\label{5sin}  \sin t = \sum_{n=0}^{\infty} \frac{(-1)^{n}}{(2n+1)!}t^{2n+1} ,\end{equation} and by the Taylor expansion of $\sin^{2} t$ about $0$, we have   \begin{equation}\label{5sin2} \sin^{2} t=-\sum_{n=1}^{\infty} \frac{(-1)^{n} 2^{-1+2n}}{(2n)!}t^{2n}.  \end{equation} 
 From Remark \ref{Remark 5.2}, we see that 
 $$\int_{I_{2}} \big|S_{f}(\frac{a}{q}+\alpha)\big|^{2} d\alpha \ll \int_{ I_{2}}\big|S_{f}(\frac{a}{q}+\alpha)\big|^{2}\big(K_{-d}(\alpha) -K_{-c}(\alpha) \big)d\alpha. $$ Therefore, we get
\begin{equation}\begin{split} \int_{ I_{2}}\big|S_{f}(\frac{a}{q}+\alpha)\big|^{2}\big(K_{-d}(\alpha) -K_{-c}(\alpha) \big)d\alpha & \leq \int_{\mathbb{R}}\big|S_{f}(\frac{a}{q}+\alpha)\big|^{2}\big|K_{-d}(\alpha) -K_{-c}(\alpha) \big|d \alpha\\&= \int_{\mathbb{R}}\big|S_{f}(\frac{a}{q}+\alpha)\big|^{2}\big(K_{-d}(\alpha) -K_{-c}(\alpha)\big)d\alpha \\& \thinspace- 2 \int_{\mathbb{R}\cap L}\big|S_{f}(\frac{a}{q}+\alpha)\big|^{2}\big(K_{-d}(\alpha) -K_{-c}(\alpha)\big)d\alpha   
\end{split}\nonumber \end{equation}
where $$L:= \{ \alpha \in \mathbb{R} : K_{-d}(\alpha)- K_{-c}(\alpha) <0\}.$$ 
By \eqref{Gallagher}, we have
\begin{equation}\label{5difference} \begin{split} \int_{\mathbb{R}}\big|S_{f}(\frac{a}{q}+\alpha)\big|^{2}\big(K_{-d}(\alpha) -K_{-c}(\alpha) \big)d\alpha &=\int_{X}^{2X} |E_{d}(x)|^{2}-|E_{c}(x)|^{2}dx\\& +O_{f,B,\epsilon}(X^{\frac{5}{6}+2\epsilon}).  \end{split}\end{equation}
%Let's consider the left-hand side of \eqref{5difference}.
%By \eqref{5difference}, we have
%\begin{equation}\begin{split}
%&\int_{I_{2} \cup L} \big|S_{f}(\frac{a}{q}+\alpha)\big|^{2}\big(K_{-d}(\alpha) -K_{-c}(\alpha) \big)d\alpha   \thinspace \thinspace \\& = \int_{I_{2} \cup (L\cap [0,\infty))}\big|S_{f}(\frac{a}{q}+\alpha)\big|^{2}\big(K_{-d}(\alpha) -K_{-c}(\alpha) \big)d\alpha \\ & \;\; + \int_{I_{2} \cup \text{-}(L\cap [0,\infty))}\big|S_{f}(\frac{a}{q}+\alpha)\big|^{2}\big(K_{-d}(\alpha) -K_{-c}(\alpha) \big)d\alpha \\&\leq \int_{\mathbb{R}} \big|S_{f}(\frac{a}{q}+\alpha)\big|^{2}\big(K_{-d}(\alpha) -K_{-c}(\alpha) \big)d\alpha \\&= \int_{X}^{2X} |E_{d}(x)|^{2}-|E_{c}(x)|^{2}dx +O_{f,\epsilon}(X^{\frac{5}{6}+2\epsilon}). \nonumber \end{split}\end{equation}
%Note that 
%\begin{equation}\begin{split}
%&\int_{L\cap [0,\infty)}
%\big|S_{f}(\frac{a}{q}+\alpha)\big|^{2}\big(K_{-d}(\alpha) -K_{-c}(\alpha) \big)d\alpha < 0,
%\\& 
%\int_{\text{-}(L\cap [0,\infty))} \big|S_{f}(\frac{a}{q}+\alpha)\big|^{2}\big(K_{-d}(\alpha) -K_{-c}(\alpha) \big)d\alpha < 0.
%\end{split} \nonumber\end{equation}
Therefore, to prove Proposition \ref{Proposition 5.1}, it is sufficient to show that 
\begin{equation}\begin{split} &\int_{L\cap [0,\infty)}\big|S_{f}(\frac{a}{q}+\alpha)\big|^{2}\big|K_{-d}(\alpha) -K_{-c}(\alpha) \big| d\alpha \ll_{f,B,\epsilon} X^{1-\frac{\epsilon}{2}},
\\& \int_{\text{-}(L\cap [0,\infty))}\big|S_{f}(\frac{a}{q}+\alpha)\big|^{2}\big|K_{-d}(\alpha) -K_{-c}(\alpha) \big| d\alpha \ll_{f,B,\epsilon} X^{1-\frac{\epsilon}{2}},\\& \int_{X}^{2X} |E_{d}(x)|^{2}-|E_{c}(x)|^{2}dx\ll_{f,B,\epsilon} X^{1-\frac{\epsilon}{2}}.\end{split}\nonumber \end{equation}  
By Proposition \ref{Proposition 5.6}, 
 \begin{equation}\label{5finish4} \int_{X}^{2X} |E_{d}(x)|^{2}-|E_{c}(x)|^{2}dx =O_{f,B,\epsilon}(X^{1-\frac{\epsilon}{2}}).
 \end{equation}
The integrals over both the positive and negative real lines can be treated with the same argument, so let us consider 
$$ \int_{L \cap [0,\infty) }\big|S_{f}(\frac{a}{q}+\alpha)\big|^{2}\big|K_{-d}(\alpha) -K_{-c}(\alpha) \big| d\alpha.$$ Note that $L\cap [0,\infty) \subset I_{1} \cup J_{1} \cup J_{3} \cup I_{4}.$
For $I_{1}=[0,\frac{1}{\pi}X^{-c}],$ by the similar argument in Remark \ref{Remark 5.2}, we have $I_{1} \cap L =\emptyset.$

\noindent Let us consider $\alpha \in J_{1}=[\frac{1}{\pi}X^{-c}, X^{-c}]\cap \{ \alpha : K_{-d}(\alpha) < K_{-c}(\alpha)\}.$ 
By the Taylor expansion of $\sin^{2} \pi X^{d} \alpha$ about $0$ \eqref{5sin2},
\begin{align} K_{-d}(\alpha) &=\Big|\frac{\sin \pi X^d \alpha}{\pi X^d \alpha}\Big|^2 \nonumber\\ &=\Big|\sum_{n=1}^{\infty} \frac{(-1)^{n+1} 2^{-1+2 n}}{(2 n) !}\left(\pi X^d \alpha\right)^{2 n-2}\Big| \nonumber\\ &=1-\frac{(\pi X^{d} \alpha)^{2}}{3}+O\left(( X^{d} \alpha)^{4}\right). \nonumber \end{align}
Since $\pi X^{d} \alpha$ goes to $0$ as $X\rightarrow \infty,$
$$K_{-d}(\alpha)  \geq 1-\frac{\pi^{2}\big(1-o(1)\big) X^{2(d-c)}}{3} \;\textrm{as} \; X \rightarrow \infty,$$ and 
$$K_{-c}(\alpha) \leq \sin^{2}\pi X^{c}\alpha$$ for $\alpha \in J_{1}.$ 
Thus, we see that 
$$m(J_{1}) \ll m\big( \{ \alpha \in [\frac{1}{\pi}X^{-c},X^{-c}]\::\: 1-\pi^{2}X^{2(d-c)} \leq \sin^{2} \pi X^{c}\alpha\}\big).$$
Since $m\big(\{t \in [0,2\pi]\::\: |\cos t| \leq \pi X^{(d-c)}\}\big)\ll X^{d-c},$ 
 $m(J_{1})\ll X^{d-2c}.$
Therefore, 
\begin{equation}\label{5finish1}\begin{split} \int_{ J_{1}}\big|S_{f}(\frac{a}{q}+\alpha)\big|^{2}\big|K_{-d}(\alpha) -K_{-c}(\alpha)\big| d\alpha&\ll_{f} X^{2}m(J_{1})\\&\ll_{f} X^{2+d-2c}\\&\ll_{f} X^{\frac{5}{6}-3\epsilon}.\end{split} \end{equation}
For $\alpha \in J_{3}= [\frac{1}{2}X^{-d},X^{\epsilon}] \cap \{ \alpha : K_{-d}(\alpha) < K_{-c}(\alpha)\},$ we have
\begin{equation}\begin{split} K_{-d}(\alpha) \leq  K_{-c}(\alpha) &\Rightarrow \Big|\frac{\sin \pi X^{d} \alpha}{\pi X^{d} \alpha}\Big|^{2} \leq \Big|\frac{1}{\pi X^{c} \alpha}\Big|^{2} \\& \Rightarrow |\sin \pi X^{d} \alpha|\leq X^{d-c}. \end{split}\nonumber  \end{equation}
Note that $\pi X^{d}\alpha \in [\frac{\pi}{2},\pi X^{d+\epsilon}]$ for $\alpha \in J_{3}.$ Let us  separate the interval $[\frac{\pi}{2}, \pi X^{d+\epsilon}]$ such that  $$[\frac{\pi}{2}, \pi X^{d+\epsilon}]= \bigcup_{j} V_{j} $$ where $V_{j}=[2j\pi ,(2j+2)\pi)$ for $j \in \mathbb{N}.$ There are $O(X^{d+\epsilon})$ sets $V_{j}$ where  $V_{j} \cap [\frac{\pi}{2}, \pi X^{d+\epsilon}] \neq \emptyset.$ 
By the Taylor expansion of $\sin t$ about $0$ \eqref{5sin}, for each $t \in V_{j},$ we have
$$|\sin t| \leq X^{d-c} \Rightarrow |t-k\pi |\leq c_{2}X^{d-c}$$ for some $k \in \{2j, 2j+1, 2j+2 \},$ constant $c_{2}>0.$
Thus, for each $j,$ we get 
$$ m\big(\{ \alpha :\: \pi X^{d}\alpha \in V_{j}, \; |\sin \pi X^{d} \alpha|< X^{d-c}\}\big) \ll X^{-c}.$$  
Therefore, $m(J_{3})\ll X^{-c+d+\epsilon}.$
Since  $|\sin x| \leq 1,$ it is easy to check that $$K_{-c}(\alpha) \ll X^{2(d-c)}$$ for $\alpha \in J_{3}.$
Since $3(d-c)=-1-3\epsilon,$ we have
\begin{equation} \begin{split} \int_{ J_{3}}\big|S_{f}(\frac{a}{q}+\alpha)\big|^{2}\big|K_{-d}(\alpha) -K_{-c}(\alpha) \big| d\alpha &\ll_{f} X^{2} m(J_{3})X^{2(d-c)}\\&\ll_{f,\epsilon} X^{1-\epsilon}.\end{split}\end{equation}
For $\alpha \in I_{4}=[X^{\epsilon}, \infty),$ we have that $\max\big(K_{-d}(\alpha), K_{-c}(\alpha)\big) \ll X^{-2d} \alpha^{-2}.$ This yields
\begin{equation}\label{5finish3}\begin{split} \int_{ I_{4}}\big|S_{f}(\frac{a}{q}+\alpha)\big|^{2}\big|K_{-d}(\alpha) -K_{-c}(\alpha)  \big| d\alpha&\ll_{f}  \int_{X^{\epsilon}}^{\infty} X^{2}X^{-2d} \frac{1}{\alpha^{2}}d\alpha\\ &\ll_{f,\epsilon} X^{1-\epsilon}.\end{split} \end{equation} 
Therefore, 
\begin{equation}\label{5finish2}
\int_{ L \cap [0,\infty) }\big|S_{f}(\frac{a}{q}+\alpha)\big|^{2}\big|K_{-d}(\alpha) -K_{-c}(\alpha) \big| d\alpha \ll_{f,\epsilon} X^{1-\epsilon}. \end{equation}
\end{proof}

\section*{\rm{Acknowledgement}}
\noindent The author would like to thank his advisor Xiaoqing Li, for suggesting this problem and giving helpful advice. The author is grateful to the referees for a careful reading of the paper and lots of invaluable suggestions. 
\nonumber

\bibliographystyle{plain}   % this means that the order of references
			    % is dtermined by the order in which the
			    % \cite and \nocite commands appear
\bibliography{over}  % list here all the bibliogres that
			     % you need. 

\end{document}